\documentclass[reqno]{amsart}
\usepackage{amssymb}
\usepackage{amsfonts}
\usepackage{amsmath}
\usepackage{color}
\usepackage{amsaddr}

\newtheorem{theorem}{Theorem}[section]

\newtheorem{corollary}[theorem]{Corollary}

\newtheorem{definition}[theorem]{Definition}

\newtheorem{lemma}[theorem]{Lemma}

\newtheorem{proposition}[theorem]{Proposition}

\theoremstyle{remark}
\newtheorem{remark}[theorem]{Remark}

\numberwithin{equation}{section}
\newcommand{\ep}{\varepsilon}

\begin{document}
\title[Massless Acoustic Perturbation]{Attractors for Damped Semilinear Wave Equations with Singularly Perturbed Acoustic Boundary Conditions}
\author[J. L. Shomberg]{Joseph L. Shomberg}
\subjclass[2010]{35B25, 35B41, 35L20, 35L71, 35Q40, 35Q70}
\keywords{Damped semilinear wave equation, acoustic boundary condition, singular perturbation, global attractor, upper-semicontinuity, exponential attractor, critical nonlinearity.}
\address{Department of Mathematics and Computer Science, Providence College, Providence, Rhode Island 02918, USA \\ {\tt{jshomber@providence.edu}} }
\date{\today}

\begin{abstract}
Under consideration is the damped semilinear wave equation 
\[
u_{tt}+u_t-\Delta u+u+f(u)=0
\]
in a bounded domain $\Omega$ in $\mathbb{R}^3$ subject to an acoustic boundary condition with a singular perturbation, which we term ``massless acoustic perturbation,''
\[
\ep\delta_{tt}+\delta_t+\delta = -u_t\quad\text{for}\quad \ep\in[0,1].
\]
By adapting earlier work by S. Frigeri, we prove the existence of a family of global attractors for each $\ep\in[0,1]$. 
We also establish the optimal regularity for the global attractors, as well as the existence of an exponential attractor, for each $\ep\in[0,1].$
The later result insures the global attractors possess finite (fractal) dimension, however, we cannot yet guarantee that this dimension is independent of the perturbation parameter $\ep.$
The family of global attractors are upper-semicontinuous with respect to the perturbation parameter $\ep$; a result which follows by an application of a new abstract result also contained in this article. 
Finally, we show that it is possible to obtain the global attractors using weaker assumptions on the nonlinear term $f$, however, in that case, the optimal regularity, the finite dimensionality, and the upper-semicontinuity of the global attractors does not necessarily hold. 
\end{abstract}

\maketitle
\tableofcontents

\section{Introduction}

Let $\Omega $ be a bounded domain in $\mathbb{R}^{3}$ with boundary $\Gamma := \partial \Omega $ of (at least) class $\mathcal{C}^{2}$. 
We consider the semilinear damped wave equation,
\begin{equation}  \label{pde}
u_{tt}+u_t-\Delta u+u+f(u)=0\quad\text{in}\quad(0,\infty )\times \Omega,
\end{equation}
with the initial conditions, 
\begin{equation}  \label{ic}
u(0,x)=u_0(x), \quad u_t(0,x)=u_1(x)\quad\text{at}\quad\{0\}\times\Omega,
\end{equation}
equipped with the singularly perturbed acoustic boundary condition,
\begin{equation}  \label{abc}
\left\{ \begin{array}{ll}
\varepsilon^2 \delta_{tt} + \delta_t + \delta = -u_t & \quad\text{on}\quad (0,\infty)\times\Gamma, \\
\delta_t = \partial_{\bf{n}}u,
\end{array} \right.
\end{equation}
where $\varepsilon\in(0,1]$, and 
\begin{equation}  \label{aic}
\delta(0,x)=\delta_0(x), \quad\varepsilon^2\delta_t(0,x) = \varepsilon^2\delta_1(x)\quad\text{at}\quad\{0\}\times\Gamma.
\end{equation}
Above, $\bf{n}$ is the outward pointing unit vector normal to the surface $\Gamma$ at $x$, and $\partial_{\bf{n}}u$ denotes the normal derivative of $u$. 
Assume the nonlinear term $f\in C^2(\mathbb{R})$ satisfies the growth condition
\begin{equation}  \label{reg-assf-2}
|f^{\prime \prime }(s)|\leq \ell (1+|s|),
\end{equation}
for some $\ell\geq 0$, and the sign condition 
\begin{equation}\label{assf-2}
\liminf_{|s|\rightarrow\infty}\frac{f(s)}{s}>-1.
\end{equation}
Also, assume that there is $\vartheta >0$ such that for all $s\in \mathbb{R}$, 
\begin{equation}  \label{reg-assf-3}
f^{\prime }(s)\geq -\vartheta.
\end{equation}
Collectively, denote the IBVP (\ref{pde})-(\ref{aic}) with (\ref{reg-assf-2})-(\ref{reg-assf-3}) as Problem (A).
The condition (\ref{dybc}) is formally obtained from (\ref{abc}) by letting $\ep=0$ and neglecting the term $\delta$ by assuming $\delta\approx0.$

Notice that the the much-studied derivative $f=F'$ of the double-well potential, 
$F(u)=\frac{1}{4}u^{4}-ku^{2}$, $k>0$, satisfies assumptions (\ref{reg-assf-2})-(\ref{reg-assf-3}). 
The first two of these assumptions, (\ref{reg-assf-2}) and (\ref{assf-2}), are the same assumptions made on the nonlinear term in \cite{CEL02}, \cite{MPZ07} and \cite{Wu&Zheng06}, for example (\cite{MPZ07} additionally assumes $f(0)=0$). The third assumption (\ref{reg-assf-3}) appears in \cite{CGG11}, \cite{Frigeri10}, \cite{Gal&Grasselli12} and \cite{Pata&Zelik06}; the bound is utilized to obtain the precompactness property for the semiflow associated with evolution equations when dynamic boundary conditions present a difficulty (e.g., here, fractional powers of the Laplace operator subject to either (\ref{abc}) or (\ref{dybc}) are undefined). 
Moreover, assumption (\ref{reg-assf-2}) implies that the growth condition for $f$ is the critical case since $\Omega\subset \mathbb{R}^{3}$. 
Such assumptions are common when one is investigating the existence of a global attractor or the existence of an exponential attractor for a partial differential equation of evolution.

Also under consideration is the ``limit problem'' where we introduce the transport-type equation as the boundary condition, 
\begin{equation}  \label{dybc}
\partial_{\bf{n}}u = -u_t\quad\text{on}\quad(0,\infty)\times\Gamma.
\end{equation}
Collectively, denote the IBVP (\ref{pde})-(\ref{ic}), (\ref{dybc}), with (\ref{reg-assf-2})-(\ref{reg-assf-3}) as Problem (T).

The damped wave equation (\ref{pde}) has frequently been studied in the context of several applications to physics, including relativistic quantum mechanics (cf. e.g. \cite{Babin&Vishik92,Temam88}).
One context for Problem (T) involves mechanical considerations in which frictional damping on the boundary $\Gamma$ is linearly proportional to the velocity $u_{t}$. 
The more general boundary condition, 
\begin{equation}  \label{dybc2}
\partial_{\bf{n}}u + u + u_t = 0 \quad\text{on}\quad(0,\infty)\times\Gamma,
\end{equation}
was recently studied in \cite{Gal&Shomberg15}.
In \cite{Wu&Zheng06}, the convergence, as time goes to infinity, of unique global strong solutions of Problem (T) to a single equilibrium is established provided that $f$ is also real analytic. 
That result is nontrivial because the set of equilibria for Problem (T) may
form a continuum. 
A version of Problem (T), but with nonlinear dissipation on the boundary, already appears in the literature, we refer to \cite{CEL02,CEL04-2,CEL04}. 
There, the authors are able to show the existence of a global attractor without the presence of the weak interior damping term $u_{t},$ by assuming that $f$ is \emph{subcritical}. 
A similar equation is studied in \cite{CLT09} with critical growth, but with localized damping present on the boundary.
The transport-type equation in the boundary condition (\ref{dybc}) also appears in \cite[Equation (1.4)]{Gal12} in the context of a Wentzell boundary condition for the heat equation.

Problem (A) describes a gas experiencing irrotational forces from a rest state in a domain $\Omega$. 
The surface $\Gamma$ acts as a locally reacting spring-like mechanism in response to excess pressure in $\Omega$. 
The unknown $\delta =\delta (t,x)$ represents the {\em{inward}} ``displacement'' of the boundary $\Gamma$ reacting to a pressure described by $-u_{t}$. 
The first equation (\ref{abc})$_{1}$ describes the spring-like effect in which $\Gamma $ (and $\delta $) interacts with $-u_{t}$, and the second equation (\ref{abc})$_{2}$ is the continuity condition: velocity of the boundary displacement $\delta $ agrees with the normal derivative of $u$. 
The presence of the term $g$ indicates nonlinear effects in the damped oscillations occurring on the surface. 
Together, (\ref{abc}) describe $\Gamma$ as a so-called locally reactive surface. 
In applications the unknown $u$ may be taken as a velocity potential of some fluid or gas in $\Omega $ that was disturbed from its equilibrium. 
The acoustic boundary condition was rigorously described by Beale and Rosencrans in \cite{Beale76,Beale&Rosencrans74}. 
Various recent sources investigate the wave equation equipped with acoustic boundary conditions, \cite{CFL01,GGG03,Mugnolo10,Vicente09}. 
However, more recently, it has been introduced as a dynamic boundary condition for problems that study the asymptotic behavior of weakly damped wave equations, see \cite{Frigeri10}.

In the case of Problem (T) and Problem (A), fractional powers of the Laplacian, which are usually utilized to decompose the solution operator into decays and compact parts, usually in pursuit to proving the existence of a global attractor, are, rather, in this context, {\em{not}} well-defined. 
The lack of fractional powers of the Laplacian means the solutions to both Problem (T) and Problem (A) cannot be obtained via a spectral basis, so local weak solutions to each problem will be obtained with semigroup methods. 
Both problems will be formulated in an abstract form and posed as an equation in a Banach space, containing a linear unbounded operator, which is the infinitesimal generator of a strongly continuous semigroup of contractions on the Banach space, and containing a locally Lipschitz nonlinear part. 

It may be of interest to the reader that the $\ep=1$ case of Problem (A) has already been studied in \cite{Frigeri10}, and it is that work, along with the recent results of \cite{Gal&Shomberg15}, that has brought the current work---in the context of a {\em{perturbation problem}}---into view. 

One of the important developments in the study of partial differential equations of evolution has been determining the stability and asymptotic behavior of the solutions. 
With these developments it has also become apparent that the stability of partial differential equations under singular perturbations has been a topic that has grown significantly; for example, we mention the continuity of attracting sets such as global attractors, exponential attractors, or (in more restrictive settings) inertial manifolds. 
We will mention only some of these important results below. 
An upper-semicontinuous family of global attractors for wave equations obtained from a perturbation of hyperbolic-relaxation type appears in \cite{Hale&Raugel88}. 
The problem is of the type
\[
\ep u_{tt}+u_t-\Delta u+\phi(u)=0,
\]
where $\ep\in [0,1]$. 
The equation possesses Dirichlet boundary conditions, and $\phi\in C^2(\mathbb{R})$ satisfies the growth assumption,
\[
\phi''(s)\leq C(1+|s|)
\]
for some $C>0$. 
The global attractor for the parabolic problem, $\mathcal{A}_0\subset H^2(\Omega)\cap H^1_0(\Omega)$, is ``lifted'' into the phase space for the hyperbolic problems, $X=H^1_0(\Omega)\times L^2(\Omega)$, by defining,
\begin{equation}  \label{HR-lift}
\mathcal{LA}_0:=\{(u,v)\in X:u\in\mathcal{A}_0,~v=f-g(u)+\Delta u\}.
\end{equation}
The family of sets in $X$ is defined by,
\begin{equation}  \label{gfam-1}
\mathbb{A}_\ep:=\left\{ \begin{array}{ll} \mathcal{L}\mathcal{A}_0 & \text{for}~\ep=0 \\ \mathcal{A}_\ep & \text{for}~\ep\in(0,1], \end{array}\right.
\end{equation}
where $\mathcal{A}_\ep\subset X$ denotes the global attractors for the hyperbolic-relaxation problem. 
The main result in \cite{Hale&Raugel88} is the upper-semicontinuity of the family of sets $\mathbb{A}_\ep$ in $X$; i.e.,
\begin{equation}  \label{4usc}
\lim_{\ep\rightarrow 0}dist_X(\mathbb{A}_\ep,\mathbb{A}_0):= \lim_{\ep\rightarrow 0}\sup_{a\in\mathbb{A}_\ep}\inf_{b\in\mathbb{A}_0}\|a-b\|_X=0.
\end{equation}
The main result in this paper is to show that a similar property holds between Problem (A) and Problem (T). 
To obtain this result, we will replace the initial conditions (\ref{aic}) with the following 
\begin{equation}  \label{aic2}
\delta(0,x) = u_0(x)+\ep\delta_0(x), \quad \ep\delta_t(0,x) = \ep\delta_1(x) \quad \text{on} \quad \{0\}\times \Gamma.
\end{equation}
Such a result insures that for every problem of type Problem (T), there is an ``acoustic relaxation'', that is Problem (A), in which (\ref{4usc}) holds. 

Since this result appeared, an upper-continuous family of global attractors for the hyperbolic-relaxation of the Cahn-Hilliard equations has been found \cite{Zheng&Milani05}. 
Robust families of exponential attractors (that is, both upper- and lower-semicontinuous with explicit control over semidistances in terms of the perturbation parameter) of the type reported in \cite{GGMP05} have successfully been demonstrated to exist in numerous applications spanning partial differential equations of evolution: the Cahn-Hilliard equations with a hyperbolic-relaxation perturbation \cite{GGMP05-CH3D,GGMP05-CH1D}, applications with a perturbation appearing in a memory kernel have been treated for reaction diffusion equations, Cahn-Hilliard equations, phase-field equations, wave equations, beam equations, and numerous others \cite{GMPZ10}. 
Recently, the existence of an upper-semicontinuous family of global attractors for a reaction-diffusion equation with a singular perturbation of hyperbolic relaxation type and {\em{dynamic}} boundary conditions has appeared in \cite{Gal&Shomberg15}.
Robust families of exponential attractors have also been constructed for equations where the perturbation parameter appears in the boundary conditions. 
Many of these applications are to the Cahn-Hilliard equations and to phase-field equations \cite{Gal08,GGM08-2,Miranville&Zelik02}.
Also, continuous families of inertial manifolds have been constructed for wave equations \cite{Mora&Morales89-2}, Cahn-Hilliard equations \cite{BGM10}, and more recently, for phase-field equations \cite{Bonfoh11}. 
Finally, for generalized semiflows and for trajectory dynamical systems (dynamical systems where well-possedness of the PDE---uniqueness of the solution, in particular---is not guaranteed), some continuity properties of global attractors have been found for the Navier-Stokes equations \cite{Ball00}, the Cahn-Hilliard equations \cite{Segatti06}, and for wave equations \cite{Ball04,Zelik04}. 

The main idea behind robustness is typically an estimate of the form, 
\begin{equation}  \label{robust-intro}
\|S_\varepsilon(t)x-\mathcal{L}S_0(t)\Pi x\|_{X_\varepsilon}\leq C\varepsilon,
\end{equation}
where $x\in X_\varepsilon$, $S_\varepsilon(t)$ and $S_0(t)$ are semigroups generated by the solutions of the perturbed problem and the limit problem, respectively, $\Pi$ denotes a projection from $X_\varepsilon$ onto $X_0$ and $\mathcal{L}$ is a ``lift'' (such as (\ref{HR-lift})) from $X_0$ into $X_\varepsilon$.
Controlling this difference in a suitable norm is crucial to obtaining our continuity result. 
The estimate (\ref{robust-intro}) means we can approximate the limit problem with the perturbation with control explicitly written in terms of the perturbation parameter. 
Usually such control is only exhibited on compact time intervals. 
It is important to realize that the lift associated with a hyperbolic-relaxation problem, for example, requires a certain degree of regularity from the limit problem. 
In particular, \cite{Gal&Shomberg15,Hale&Raugel88} rely on (\ref{HR-lift}); so one needs $\mathcal{A}_0\subset H^2$ in order for $\mathcal{LA}_0\subset L^2$ to be well-defined.
For the model problem presented here, the perturbation parameter $\ep$ only appears in the (dynamic) boundary condition. 
For the model problem under consideration here, the perturbation is singular in nature, however, additional regularity from the global attractor $\mathcal{A}_0$ is not required in order for the {\em{lift}} to be well-defined. 
However, additional regularity, guaranteed by assumptions (\ref{reg-assf-2})-(\ref{reg-assf-3}), will be required in order to achieve an estimate like (\ref{robust-intro}).
The regularity of the attractor $\mathcal{A}_0$ is instead needed to control the difference in (\ref{robust-intro}); in this way we prove the upper-semicontinuity of the family of global attractors. 

Unlike global attractors described above, exponential attractors (sometimes called, inertial sets) are positively invariant sets possessing finite fractal dimension that attract bounded subsets of the phase space exponentially fast. 
It can readily be seen that when both a global attractor $\mathcal{A}$ and an exponential attractor $\mathcal{M}$ exist, then $\mathcal{A}\subseteq \mathcal{M}$, and so the global attractor is also finite dimensional. 
When we turn our attention to proving the existence of exponential attractors, certain higher-order dissipative estimates are required. 
In the case for Problem (T) and Problem (A), the estimates cannot be obtained along the lines of multiplication by fractional powers of the Laplacian; as we have already described, we need to resort to other methods.
In particular, we will apply $H^2$-elliptic regularity methods as in \cite{Pata&Zelik06}. 
Here, the main idea is to differentiate the equations with respect to time $t$ to obtain uniform estimates for the new equations. 
This strategy has recently received a lot of attention.
Some successes include dealing with a damped wave equation with acoustic boundary conditions \cite{Frigeri10} and a wave equation with a nonlinear dynamic boundary condition \cite{CEL02,CEL04-2,CEL04}.
Also, there is the hyperbolic relaxation of a Cahn-Hilliard equation with dynamic boundary conditions \cite{CGG11,Gal&Grasselli12}.
Additionally, this approach was also taken in \cite{Gal&Shomberg15}.
The drawback from using this approach comes from the difficulty in finding appropriate estimates that are {\em{uniform}} in the perturbation parameter $\ep.$
Indeed, this was the case in \cite{Gal&Shomberg15}.
There, the authors we able to find an upper-semicontinuous family of global attractors and a family of exponential attractors. 
It turned out that a certain higher-order dissipative estimate depends on $\ep$ in a crucial way, and consequently, the robustness/H\"older continuity of the family of exponential attractors cannot (yet) be obtained.
Furthermore, as it turns out, the global attractors found in \cite{Gal&Shomberg15} have finite (fractal) dimension, although the dimension is not necessarily independent of $\ep.$ 
It appears that similar difficulties persist with the model problem examined here.

The main results in this paper are:
\begin{itemize}
\item An upper-semicontinuity result for a {\em{generic}} family of sets for a family of semiflows, where in particular, the limit ($\ep=0$) semigroup of solution operators is locally Lipschitz continuous, uniformly in time on compact intervals. 

\item Problem (T) and Problem (A) admit a family of global attractors $\{\mathcal{A}_\ep\}_{\ep\in[0,1]}$, bounded, uniformly in $\ep\in[0,1]$, in the respective phase space. 
The global attractors possess optimal regularity and are bounded in a more regular phase space, however this bound is {\em{not}} independent of $\ep$. 

\item The generic semicontinuity result is applied to the family of global attractors $\{\mathcal{A}_\ep\}_{\ep\in[0,1]}$. 
The result shows the family of global attractors is upper-semicontinuous.

\item There exists a family of exponential attractors $\{\mathcal{M}_\ep\}_{\ep\in[0,1]}$, admitted by the semiflows associated with for Problem (T) and Problem (A).
Since $\mathcal{A}_\ep\subset\mathcal{M}_\ep$ for each $\ep\in[0,1]$, this result insures the global attractors inherit finite (fractal) dimension.
However, we cannot conclude that dimension is uniform in $\ep$ (this result remains open).

\item We also show the existence of the global attractors under weaker assumptions on the nonlinear term $f$. 
Although the attractor $\mathcal{A}_0$ may be embedded/lifted into the phase space for the perturbation problem with no further regularity needed, the various other properties earned from the regularity---optimal regularity, upper-semicontinuity, and finite dimensionality---no longer hold. 

\end{itemize}

{\textbf{Notation and conventions}}. 
We take the opportunity here to introduce some notations and conventions that are used throughout the paper. 
We denote by $\Vert \cdot \Vert $, $\Vert \cdot \Vert _{k}$, the norms in $L^{2}(\Omega )$, $H^{k}(\Omega )$, respectively. 
We use the notation $\langle \cdot ,\cdot \rangle $ and $\langle \cdot ,\cdot \rangle_{k}$, $k\ge1$, to denote the products on $L^{2}(\Omega )$ and $H^{k}(\Omega)$, respectively.
For the boundary terms, $\Vert \cdot \Vert _{L^{2}(\Gamma )}$ and $\langle
\cdot ,\cdot \rangle _{L^{2}(\Gamma )}$ denote the norm and, respectively,
product on $L^{2}(\Gamma )$. 
We will require the norm in $H^{k}(\Gamma )$, to be denoted by $\Vert \cdot \Vert _{H^{k}(\Gamma )}$, where $k\geq 1$. 
The $L^{p}(\Omega )$ norm, $p\in (0,\infty ]$, is denoted $|\cdot |_{p}$. 
The dual pairing between $H^{1}(\Omega )$ and the dual $H^{-1}(\Omega) := (H^{1}(\Omega))^*$ is denoted by $(u,v)_{H^{-1}\times H^1}$. 
In many calculations, functional notation indicating dependence on the variable $t$ is dropped; for example, we will write $u$ in place of $u(t)$. 
Throughout the paper, $C>0$ will denote a \emph{generic} constant which may depend various structural constants, while $Q:\mathbb{R}_{+}\rightarrow \mathbb{R}_{+}$ will denote a \emph{generic} increasing function. 
All these quantities, unless explicitly stated, are \emph{independent} of the perturbation parameter $\varepsilon$.
Further dependencies of these quantities will be specified on occurrence.
We will use $\|B\|_{W}:=\sup_{\Upsilon\in B}\|\Upsilon\|_W$ to denote the ``size'' of the subset $B$ in the Banach space $W$.
Let $\lambda >0$ be the best Sobolev--Poincar\'{e} type constant 
\begin{equation}  \label{Poincare}
\lambda \int_{\Omega}u^{2}dx\leq \int_{\Omega }|\nabla u|^{2}
dx+\int_{\Gamma }u^{2}d\sigma.
\end{equation}
Later in the article, we will rely on the Laplace--Beltrami operator $-\Delta_\Gamma$ on the surface $\Gamma.$ This operator is positive definite and self-adjoint on $L^2(\Gamma)$ with domain $D(-\Delta_\Gamma)$.
The Sobolev spaces $H^s(\Gamma)$, for $s\in\mathbb{R}$, may be defined as $H^s(\Gamma)=D((-\Delta_\Gamma)^{s/2})$ when endowed with the norm whose square is given by, for all $u\in H^s(\Gamma)$,
\begin{equation}  \label{LB-norm}
\|u\|^2_{H^s(\Gamma)} := \|u\|^2_{L^2(\Gamma)} + \left\|(-\Delta_\Gamma)^{s/2}u\right\|^2_{L^2(\Gamma)}.
\end{equation}

The plan of the paper: in Section 2 we review the important results concerning the limit ($\ep=0$) Problem (T), and in Section 3 we discuss the relevant results concerning the perturbation Problem (A). 
Some important remarks describing several instances of how Problem (A) depends of the perturbation parameter $\ep>0$ are mentioned in Section 3. 
The final Section 4 contains a new abstract upper-semicontinuity result that is then tailored specifically for the model problem under consideration. 

\section{Attractors for Problem (T), the $\ep=0$ case}  \label{s:dynamic}

In this section, we review Problem (T). 
The well-posedness of Problem (T), as well as the existence of a global attractor and an exponential attractor was already established in the work of \cite{Gal&Shomberg15}. 

The finite energy phase space for the problem is the space 
\begin{equation*}
\mathcal{H}_0=H^{1}(\Omega )\times L^{2}(\Omega).
\end{equation*}
The space $\mathcal{H}_0$ is Hilbert when endowed with the norm whose square is given by, for $\varphi =(u,v)\in \mathcal{H}_0 = H^{1}(\Omega )\times L^{2}(\Omega )$, 
\begin{align}
\Vert \varphi \Vert _{\mathcal{H}_0}^{2} & := \Vert u\Vert_{1}^{2} + \Vert v\Vert ^{2}   \notag \\
& = \left( \Vert \nabla u\Vert ^{2} + \Vert u\Vert^{2} \right) + \Vert v\Vert ^{2}.  \notag 
\end{align}

We will denote by $\Delta _{N}:L^{2}(\Omega )\rightarrow L^{2}(\Omega )$ the homogeneous Neumann--Laplacian operator with domain 
\begin{equation*}
D(\Delta _{N})=\{u\in H^{2}(\Omega ):\partial _{\mathbf{n}}u=0~
\text{on}~\Gamma \}.
\end{equation*}
Of course, the operator $-\Delta _{N}$ is self-adjoint and positive. 
The Neumann--Laplacian is extended to a continuous operator $\Delta _{N}:H^{1}(\Omega )\rightarrow H^{-1}(\Omega )$, defined by, for all $v\in H^{1}(\Omega )$, 
\begin{equation*}
(-\Delta _{N}u,v)_{H^{-1}\times H^1}=\langle \nabla u,\nabla v\rangle.
\end{equation*}
Motivated by \cite{CEL02,Wu&Zheng06}, we define the ``Neumann'' map $N : H^{s}(\Gamma)\rightarrow H^{s+(3/2)}(\Omega )$ by 
\begin{equation*}
Np=q~\text{if and only if}~\Delta q=0~\text{in}~\Omega ,~\text{and}~\partial_{\mathbf{n}}q = p~\text{on}~\Gamma .
\end{equation*}
The adjoint of the Neumann map satisfies, for all $v\in H^{1}(\Omega )$, 
\begin{equation*}
N^*\Delta _{N}v = -v~\text{on}~\Gamma .
\end{equation*}

Define the closed subspace of $H^{2}(\Omega )\times H^{1}(\Omega )$, 
\begin{equation*}
\mathcal{D}_0 := \{(u,v)\in H^{2}(\Omega )\times H^{1}(\Omega
):\partial _{\mathbf{n}}u = -v~\text{on}~\Gamma \}.
\end{equation*}
endowed with norm whose square is given by, for all $\varphi =\left( u, v \right) \in \mathcal{D}_0$, 
\begin{equation*}
\Vert \varphi \Vert _{\mathcal{D}_0}^{2} := \Vert u\Vert
_{2}^{2} + \Vert v\Vert _{1}^{2}.
\end{equation*}
Let $D(A_0) = \mathcal{D}_0$. Define the linear
unbounded operator $A_0 : D(A_0) \rightarrow \mathcal{
H}_0$ by 
\begin{equation*}
A_0 := \begin{pmatrix}
0 & 1 \\ 
\Delta_{N}-1 & \Delta_{N}N~tr_{D}(\cdot)-1
\end{pmatrix},
\end{equation*}
where $tr_D :H^s(\Omega)\rightarrow H^{s-1/2}(\Gamma)$, $s>\frac{1}{2}$, denotes the Dirichlet trace operator (i.e., $tr_{D}(v) = v_{\mid\Gamma}$). 
Notice that if $(u,v)\in \mathcal{D}_0$, then $u+Ntr_{D}(v)\in D(\Delta_{N})$. 
By the Lumer--Phillips theorem (cf., e.g., \cite[Theorem I.4.3]{Pazy83}) and the Lax--Milgram theorem, it is not hard to see that the operator $A_0$, with domain $\mathcal{D}_0$, is an infinitesimal generator of a strongly continuous semigroup of contractions on $\mathcal{H}_0$, denoted $e^{A_0t}$.

Define the map $\mathcal{F}_0:\mathcal{H}_0\rightarrow \mathcal{H}_0$ by 
\begin{equation*}
\mathcal{F}_0(\varphi ):= \begin{pmatrix}
0 \\ 
-f(u) \end{pmatrix}
\end{equation*}
for all $\varphi = (u,v)\in \mathcal{H}_0$. Since $f:H^{1}(\Omega )\rightarrow L^{2}(\Omega )$ is locally Lipschitz continuous 
\cite[cf., e.g., Theorem 2.7.13]{Zheng04}, it follows that the map $\mathcal{F}_0:\mathcal{H}_0 \rightarrow \mathcal{H}_0$ is as well.

Problem (T) may be put into the abstract form in $\mathcal{H}_0$, for $\varphi(t)=(u(t),u_{t}(t))$,
\begin{equation}  \label{abstract-d}
\displaystyle\frac{d}{dt}\varphi (t) = A_0\varphi (t) + \mathcal{F}_0(\varphi (t));\quad\varphi (0)=
\begin{pmatrix}
u_{0} \\ u_{1}
\end{pmatrix}.
\end{equation}

\begin{lemma}  \label{adjoint-d}
The adjoint of $A_0$, denoted $A_0^*$, is given by 
\begin{equation*}
A_0^*:= - \begin{pmatrix}
0 & 1 \\ 
\Delta _{N}-1 & -(\Delta_{N}N~tr_{D}(\cdot)-1),
\end{pmatrix}
\end{equation*}
with domain 
\begin{equation*}
D(A_0^*) := \{ (\chi ,\psi ) \in H^{2}(\Omega )\times
H^{1}(\Omega ):\partial _{\mathbf{n}}\chi = - \psi ~\text{on}~\Gamma \}.
\end{equation*}
\end{lemma}

\begin{proof}
The proof is a calculation similar to, e.g., \cite[Lemma 3.1]{Ball04}.
\end{proof}

Formal multiplication of the PDE (\ref{pde}) by $2u_t$ in $L^2(\Omega)$ produces the {\em{energy equation}}
\begin{equation}\label{Robin-energy-3}
\frac{d}{dt} \left\{ \|\varphi\|^2_{\mathcal{H}_0} + 2\int_\Omega F(u) dx \right\} + 2\|u_t\|^2 = 0.
\end{equation}
Here, $F(s)=\int_0^s f(\sigma) d\sigma$. 

The following inequalities are straight forward consequences of the Poincar\'{e}-type inequality (\ref{Poincare}) and assumption (\ref{assf-2}), there is a constant $\mu_0\in(0,1]$ such that, for all $u\in H^1(\Omega)$,
\begin{equation}  \label{consf-1}
2\int_\Omega F(u) dx \geq -(1-\mu_0)\|u\|^2_1 - \kappa_f
\end{equation}
for some constant $\kappa_f \ge 0$.
A proof of (\ref{consf-1}) can be found in \cite[page 1913]{CEL02}.
Furthermore, with (\ref{reg-assf-3}) and integration by parts on $F(s)=\int_{0}^{s}f(\sigma )d\sigma $, we have the upper-bound 
\begin{align}  \label{consf-2}
\int_\Omega F(\xi)dx & \leq \langle f(\xi),\xi \rangle + \frac{\vartheta}{2\lambda}\|\xi\|^2_1.
\end{align}
Moreover, the inequality
\begin{equation}  \label{consf-3}
\langle f(u),u \rangle \geq -(1-\mu_0)\|u\|^2 - \kappa_f
\end{equation}
follows from the sign condition (\ref{assf-2}) where $\mu_0\in(0,1]$ and $\kappa_f\geq0$ are from (\ref{consf-1}). 

The notion of weak solution to Problem (T) is as follows (see, \cite{Ball77}).

\begin{definition}
Let $T>0$ and $(u_{0},u_{1})\in \mathcal{H}_0$.
A map $\varphi = (u,u_{t})\in C([0,T];\mathcal{H}_0)$ is a weak solution of (\ref{abstract-d}) on $[0,T],$ if for each $\theta =(\chi ,\psi )\in D(A_0^*)$ the map $t\mapsto \langle \varphi (t),\theta \rangle _{\mathcal{H}_0}$ is absolutely continuous on $[0,T]$ and satisfies, for almost all $t \in \lbrack 0,T]$, 
\begin{equation}  \label{abs-d-1}
\frac{d}{dt}\langle \varphi (t),\theta \rangle _{\mathcal{H}_0} = \langle \varphi (t), A_0^*\theta \rangle_{\mathcal{H}_0} + \langle \mathcal{F}_0(\varphi (t)),\theta \rangle _{\mathcal{H}_0}.  
\end{equation}
The map $\varphi =(u,u_{t})$ is a weak solution on $[0,\infty )$ (i.e., a \emph{global weak solution}) if it is a weak solution on $[0,T]$, for all $T>0$.
\end{definition}

According to \cite[Definition 3.1 and Proposition 3.5]{Ball04}, the notion of weak solution above is equivalent to the following notion of a mild solution.

\begin{definition}
A function $\varphi = (u,u_{t}):[0,T] \rightarrow \mathcal{H}
_0$ is a weak solution of (\ref{abstract-d}) on $[0,T],$ if and only if $\mathcal{F}_0(\varphi (\cdot ))\in L^{1}(0,T;\mathcal{H}_0)$ and $\varphi$ satisfies the variation of constants formula, for all $t\in \lbrack 0,T]$, 
\begin{equation*}
\varphi (t) = e^{A_0t}\varphi _{0}+\int_{0}^{t}e^{A_0(t-s)}\mathcal{F}_0(\varphi (s))ds.
\end{equation*}
\end{definition}

Furthermore, by \cite[Proposition 3.4]{Ball04} and the explicit characterization of $D(A_0^*)$, our notion of weak solution is also equivalent to the standard concept of a weak (distributional) solution to Problem (T).

\begin{definition}  \label{d:weak-d}
A function $\varphi = (u,u_{t}):[0,T]\rightarrow \mathcal{H}_0$ is a weak solution of (\ref{abstract-d}) on $[0,T],$ if
\begin{equation*}
\varphi =(u,u_{t})\in C(\left[ 0,T\right] ;\mathcal{H}_0), \quad u_{t}\in L^{2}(\left[ 0,T\right] \times \Gamma ),
\end{equation*}
and, for each $\psi \in H^{1}\left( \Omega \right) ,$ $\left( u_{t},\psi \right) \in C^{1}\left( \left[ 0,T\right] \right) $ with
\begin{equation}  \label{weakf}
\frac{d}{dt} \left\langle u_t(t) ,\psi
\right\rangle + \left\langle \nabla u\left( t\right) ,\nabla \psi \right\rangle + \left\langle u_{t}\left( t\right) ,\psi \right\rangle + \left\langle u_{t}\left( t\right) + u\left( t\right) ,\psi \right\rangle _{L^{2}\left( \Gamma \right) } = -\left\langle f\left( u\left( t\right) \right) ,\psi \right\rangle ,  
\end{equation}
for almost all $t\in \left[ 0,T\right] .$
\end{definition}

Indeed, by \cite[Lemma 3.3]{Ball04} we have that $f:H^{1}\left( \Omega \right) \rightarrow L^{2}\left( \Omega \right) $ is sequentially weakly continuous and continuous, on account of the assumptions (\ref{reg-assf-2})-(\ref{assf-2}). Moreover, by \cite[Proposition 3.4]{Ball04} and the explicit representation of $D(A_0^*)$ in Lemma \ref{adjoint-d}, $\langle \varphi_{t},\theta \rangle \in C^{1}([0,T])$ for all $\theta \in D(A_0^*)$, and (\ref{abs-d-1}) is satisfied. 

Finally, the following notion of strong solution to Problem (T) is as follows.

\begin{definition}  \label{d:strong-d} 
Let $\varphi _{0} = \left( u_{0},u_{1}\right) \in \mathcal{D}_0$, i.e., $(u_{0},u_{1})\in H^{2}(\Omega )\times H^{1}(\Omega )$ such that it satisfies the compatibility condition
\begin{equation*}
\partial _{\mathbf{n}}u_{0} = -u_{1},~\text{on}~\Gamma .
\end{equation*}
A function $\varphi \left( t\right) = \left( u\left( t\right),u_{t}\left( t\right) \right) $ is called a (global) strong
solution if it is a weak solution in the sense of Definition \ref{d:weak-d}, and if it satisfies the following regularity properties:
\begin{equation}
\begin{array}{l}  \label{reg-prop}
\varphi \in L^{\infty }([0,\infty) ;\mathcal{D}_0), \quad \varphi_{t}\in L^{\infty }([0,\infty) ;\mathcal{H}_0), \\ 
~u_{tt}\in L^{\infty }([0,\infty) ;L^{2}(\Omega )), \quad u_{tt}\in L^{2}([0,\infty) ;L^{2}(\Gamma )).
\end{array}
\end{equation}
Therefore, $\varphi(t)=(u(t),u_{t}(t))$ satisfies the equations (\ref{pde}), (\ref{ic}), (\ref{dybc}) almost everywhere, i.e., is a strong solution.
\end{definition}

The following results are due to \cite{Gal&Shomberg15}. 

\begin{theorem}  \label{t:weak-soln} 
Assume (\ref{reg-assf-2}) and (\ref{assf-2}) hold. 
For each $\varphi _{0}=(u_{0},u_{1})\in \mathcal{H}_0$, there exists a unique global weak solution $\varphi = (u,u_{t}) \in C([0,\infty );\mathcal{H}_0)$ to Problem (T). 
In addition, 
\begin{equation}  \label{bndry-a-1}
\partial _{\mathbf{n}} u \in L_{loc}^{2}([0,\infty )\times \Gamma )\quad\text{and}\quad u_{t}\in L_{loc}^{2}([0,\infty )\times \Gamma ).
\end{equation}
For each weak solution, the map 
\begin{equation}  \label{C1-map}
t \mapsto \Vert \varphi (t)\Vert _{\mathcal{H}_0}^{2} + 2\int_{\Omega } F_0(u(t)) dx
\end{equation}
is $C^{1}([0,\infty ))$ and the energy equation
\begin{equation}  \label{energy-1}
\frac{d}{dt} \left\{ \Vert \varphi (t)\Vert _{\mathcal{H}_0}^{2} + 2\int_{\Omega } F(u(t)) dx\right\} = -2\Vert u_{t}(t)\Vert ^{2} - 2\Vert u_{t}(t)\Vert _{L^{2}(\Gamma )}^{2}
\end{equation}
holds (in the sense of distributions) a.e. on $[0,\infty )$. Furthermore, let $\varphi (t)=(u(t),u_{t}(t))$ and $\theta (t)=(v(t),v_{t}(t))$ denote the corresponding weak solution with initial data $\varphi_{0}=(u_{0},u_{1})\in \mathcal{H}_0$ and $\theta_{0}=(v_{0},v_{1})\in \mathcal{H}_0$, respectively, such that $\left\Vert \varphi _{0}\right\Vert _{\mathcal{H}_0}\leq R,$ $\left\Vert \theta _{0}\right\Vert _{\mathcal{H}_0}\leq R.$ Then there exists a constant $\nu_{0}=\nu_{0}(R)>0$, such that, for all $t\geq 0 $,
\begin{align}
& \|\varphi (t)-\theta (t)\|_{\mathcal{H}_0}^{2} +\int_{0}^{t} \left( \|u_{t}(\tau)-v_{t}(\tau)\|^{2} + \|u_{t}(\tau) -v_{t}(\tau)\|_{L^{2}(\Gamma)}^{2} \right) d\tau  \label{cont-dep} \\
& \leq e^{\nu_{0}t}\Vert \varphi _{0} - \theta _{0}\Vert _{\mathcal{H}_0}^{2}.  \notag
\end{align}
Furthermore, when (\ref{reg-assf-3}) also holds, for each $(u_{0},u_{1})\in \mathcal{D}_0$, Problem (T) possesses a unique global strong solution in the sense of Definition \ref{d:strong-d}.
\end{theorem}

In view of Theorem \ref{t:weak-soln}, the following result which allows us to define a dynamical system on $\mathcal{H}_0$ is immediate.

\begin{corollary}  \label{sf-r}
Let the assumptions of Theorem \ref{t:weak-soln} be satisfied. 
Let $\varphi_0=(u_0,u_1)\in\mathcal{H}_0$ and $u$ be the unique global solution of Problem (T). 
The family of maps $S_0=(S_0(t))_{t\geq 0}$ defined by 
\[
S_0(t)\varphi_0(x):=(u(t,x,u_0,u_1),u_t(t,x,u_0,u_1))
\]
is a {\em{semiflow}} generated by Problem (T).
The operators $S_0(t)$ satisfy
\begin{enumerate}
	\item $S_0(t+s)=S_0(t)S_0(s)$ for all $t,s\geq 0$.
	\item $S_0(0)=I_{\mathcal{H}_0}$ (the identity on $\mathcal{H}_0$)
	\item $S_0(t)\varphi_0\rightarrow S_0(t_0)\varphi_0$ for every $\varphi_0\in\mathcal{H}_0$ when $t\rightarrow t_0$.
\end{enumerate}

Additionally, each mapping $S_0(t):\mathcal{H}_0\rightarrow\mathcal{H}_0$ is Lipschitz continuous, uniformly in $t$ on compact intervals; i.e., for all $\varphi_0, \theta_0\in\mathcal{H}_0$, and for each $T\geq 0$, and for all $t\in[0,T]$,
\begin{equation}  \label{S0-Lipschitz-continuous}
\|S_0(t)\varphi_0-S_0(t)\theta_0\|_{\mathcal{H}_0} \leq e^{\nu_0 T}\|\varphi_0-\theta_0\|_{\mathcal{H}_0}.
\end{equation}
\end{corollary}

\begin{proof}
The semigroup properties (1) and (2) are well-known and apply to a general class of abstract Cauchy problems possessing many applications (see \cite{Babin&Vishik92,Morante79,Goldstein85,Tanabe79}; in particular, a proof of property (1) is given in \cite[\S1.2.4]{Milani&Koksch05}). 
The continuity in $t$ described by property (3) follows from the definition of weak solution (this also establishes strong continuity of the operators when $t_0=0$). 
The continuity property (\ref{S0-Lipschitz-continuous}) follows from (\ref{cont-dep}).
\end{proof}

We will now show that the dynamical system $(S_0(t),\mathcal{H}_0)$ generated by the weak solutions of Problem (T) is dissipative in the sense that $S_0$ admits a closed, positively invariant, bounded absorbing set in $\mathcal{H}_0$. 

\begin{lemma}  \label{t:abs-set-d} 
Assume (\ref{reg-assf-2}) and (\ref{assf-2}) hold. 
For all $\varphi _{0}=(u_{0},u_{1})\in \mathcal{H}_0$, there exist a positive function $Q$ and constants $\omega _{0}>0$, $P_{0}>0$, such that $\varphi (t)$ satisfies, for all $t\geq 0$, 
\begin{equation}  \label{decay-1}
\|\varphi (t)\|_{\mathcal{H}_0}^{2} \leq Q(\Vert \varphi _{0}\Vert _{\mathcal{H}_0}) e^{-\omega _{0}t} + P_{0}.  
\end{equation}
Consequently, the ball $\mathcal{B}_0$ in $\mathcal{H}_0$, 
\begin{equation}  \label{abs-set-d}
\mathcal{B}_0 := \{\varphi \in \mathcal{H}_0:\Vert
\varphi \Vert _{\mathcal{H}_0} \leq P_{0}+1 \}
\end{equation}
is a bounded absorbing set in $\mathcal{H}_0$ for the dynamical system $(S_0(t),\mathcal{H}_0).$
\end{lemma}

\begin{theorem}  \label{t:exp-attr-d} 
Assume (\ref{reg-assf-2}), (\ref{assf-2}), and (\ref{reg-assf-3}) hold. 
There exists $\omega_1>0$ and a closed and bounded subset $\mathcal{C}_0\subset \mathcal{D}_0$, such that for every nonempty bounded subset $B\subset \mathcal{H}_0$,
\begin{equation}  \label{trans-d}
dist_{\mathcal{H}_0}(S_0(t)B,\mathcal{C}_0) \leq Q(\left\Vert B\right\Vert _{\mathcal{H}_0}) e^{-\omega_1 t}. 
 \end{equation}
\end{theorem}

By standard arguments of the theory of attractors (see, e.g., \cite{Hale88,Temam88}), the existence of a compact global attractor $\mathcal{A}_0 \subset \mathcal{C}_0$ for the semigroup $S_0(t)$ follows.

\begin{theorem}  \label{t:gattr-d} 
Let the assumptions of Theorem \ref{t:exp-attr-d} hold. 
The semiflow $S_0$ generated by the solutions of Problem (T) admits a unique global attractor
\begin{equation}  \label{gattr-d-2}
\mathcal{A}_0 = \omega (\mathcal{B}_0) := \bigcap_{s\geq t}{\overline{\bigcup_{t\geq 0}S_0(t)\mathcal{B}_0}}^{\mathcal{H}_0}
\end{equation}
in $\mathcal{H}_0$. Moreover, the following hold:

(i) For each $t\geq 0$, $S_0(t)\mathcal{A}_0 = \mathcal{A}_0$.

(ii) For every nonempty bounded subset $B$ of $\mathcal{H}_0$,
\begin{equation}  \label{gattraction}
\lim_{t\rightarrow \infty }dist_{\mathcal{H}_0}(S_0(t)B,\mathcal{A}_0) = 0.
\end{equation}

(iii) The global attractor $\mathcal{A}_0$ is bounded in $\mathcal{D}_0$ and trajectories on $\mathcal{A}_0$ are strong solutions.
\end{theorem}

The existence of an exponential attractor follows from the application of the abstract result (see, e.g., \cite[Proposition 1]{EMZ00}, \cite{FGMZ04}, \cite{GGMP05}).

\begin{theorem}  \label{t:expattr-d} 
Assume (\ref{reg-assf-2}), (\ref{assf-2}), and (\ref{reg-assf-3}) hold. 
The dynamical system $\left( S_0,\mathcal{H}_0\right)$
associated with Problem (T) admits an exponential attractor $\mathcal{M}_0$ compact in $\mathcal{H}_0$, and bounded in $\mathcal{C}_0$. Moreover, there hold:

(i) For each $t\geq 0$, $S_0(t)\mathcal{M}_0 \subseteq \mathcal{M}_0$.

(ii) The fractal dimension of $\mathcal{M}_0$ with respect to
the metric $\mathcal{H}_0$ is finite, namely,
\begin{equation*}
\dim _{\rm{F}}\left( \mathcal{M}_0,\mathcal{H}_0 \right) \leq C <\infty,
\end{equation*}
for some positive constant $C$.

(iii) There exist $\varrho_0>0$ and a positive nondecreasing function $Q$ such that, for all $t\geq 0$, 
\begin{equation*}
dist_{\mathcal{H}_0}(S_0(t)B,\mathcal{M}_0) \leq Q(\|B\|_{\mathcal{H}_0}) e^{-\varrho_0 t},
\end{equation*}
for every nonempty bounded subset $B$ of $\mathcal{H}_0.$
\end{theorem}

\begin{remark}
Above,
\begin{equation*}
\dim _{F}(\mathcal{M}_0,\mathcal{H}_0) := \limsup_{r\rightarrow 0} \frac{\ln \mu _{\mathcal{H}_0}(\mathcal{M}_0,r)}{-\ln r} < \infty ,
\end{equation*}
where, $\mu _{\mathcal{H}_0}(\mathcal{X},r)$ denotes the
minimum number of $r$-balls from $\mathcal{H}_0$ required to
cover $\mathcal{X}$.
\end{remark}

\begin{corollary}
Under the assumptions of Theorem \ref{t:expattr-d}, there holds
\begin{equation*}
\dim _{F}(\mathcal{A}_0,\mathcal{H}_0) \leq \dim _{F}(\mathcal{M}_0, \mathcal{H}_0).
\end{equation*}
As a consequence, $\mathcal{A}_0$ has finite fractal dimension.
\end{corollary}

\section{Attractors for Problem (A), the $\ep>0$ case}  \label{s:acoustic}

In this section Problem (A) is discussed. 
Weak solutions, dissipativity (i.e., the existence of an absorbing set), as well as the existence of a global attractor in this case was established under the assumptions (\ref{reg-assf-2})-(\ref{assf-2}) in \cite{Frigeri10} for the case when $\ep=1$.
Strong solutions are shown to exist under the assumptions (\ref{reg-assf-2}), (\ref{assf-2}), and (\ref{reg-assf-3}); further, the optimal regularity of the global attractor and the existence of an exponential attractor we also shown in \cite{Frigeri10} when $\ep=1$.
The well-posedness and dissipativity results stated in this section follow directly from \cite{Frigeri10} after modifications to incorporate the perturbation $0<\varepsilon\leq 1$. 
Indeed, the main results presented here follow directly from \cite{Frigeri10} with suitable modifications to account for the perturbation parameter $\ep$ occurring in the equation governing the acoustic boundary condition. 
We do not present all the proofs for the case $\ep\in(0,1)$ since the modified proofs follow directly from Frigeri's work \cite{Frigeri10} with only minor modifications, but in some instances $\ep$ may appear in a crucial way in some parameters.
In particular, we do present the proof for the existence of an absorbing set to demonstrate the independence of various parameters (such as the time of entry into the absorbing set and the radius of the absorbing set) of the perturbation parameter $\ep$.
Other observations will be explained, where needed, by a remark following the statement of the claim.

By using the same arguments in \cite{Frigeri10}, it can easily be shown that, for each $\ep\in(0,1]$, Problem (A) possesses unique global weak solutions in a suitable phase space, and the solutions depend continuously on the initial data. 
For the reader's convenience, we sketch the main arguments involved in the proofs. 
As with Problem (T), the solutions generate a family of Lipschitz continuous semiflows, now depending on $\ep$, each of which admits a bounded, absorbing, positively invariant set. 
As mentioned, we will also establish the existence of a family of exponential attractors.
Furthermore, under assumptions (\ref{reg-assf-2}), (\ref{assf-2}), (\ref{reg-assf-3}), we also show the existence of a family of exponential attractors, however, any robustness/H\"older continuity result for the family of exponential attractors is still out of reach. 
The upper-semicontinuity of the family of global will be shown in Section \ref{s:cont}.

The phase space and abstract formulation for the perturbation problem is now discussed. 
In contrast to the previous section, the underlying spaces, maps and operators now depend on the perturbation parameter $\ep$.
Let 
\[
\mathcal{H}:= H^1(\Omega) \times L^2(\Omega) \times L^2(\Gamma) \times L^2(\Gamma).
\]
The space $\mathcal{H}$ is Hilbert with the norm whose square is given by, for $\zeta=(u,v,\delta,\gamma)\in\mathcal{H}$,
\[
\|\zeta\|^2_{\mathcal{H}} := \|u\|^2_1 + \|v\|^2 + \|\delta\|^2_{L^2(\Gamma)} + \|\gamma\|^2_{L^2(\Gamma)}.
\]
Let $\varepsilon>0$ and denote by $\mathcal{H}_\varepsilon$ the space $\mathcal{H}$ when endowed with the $\varepsilon$-weighted norm whose square is given by
\begin{align}
\|\zeta\|^2_{\mathcal{H}_\varepsilon} & := \|u\|^2_1 + \|v\|^2 + \|\delta\|^2_{L^2(\Gamma)} + \varepsilon^2 \|\gamma\|^2_{L^2(\Gamma)}   \notag \\ 
& = \left( \|\nabla u\|^2 + \|u\|^2 \right) + \|v\|^2 + \|\delta\|^2_{L^2(\Gamma)} + \varepsilon^2 \|\gamma\|^2_{L^2(\Gamma)}.  \notag
\end{align}
Let
\[
D(\Delta):= \{ u\in L^2(\Omega) : \Delta u\in L^2(\Omega) \},
\]
and define the set
\[
D(A_\ep) := \left\{ (u,v,\delta,\gamma)\in D(\Delta) \times H^1(\Omega) \times L^2(\Gamma) \times L^2(\Gamma) : \partial_{\bf{n}}u = \gamma \ \text{on} \ \Gamma \right\}.
\]
Define the linear unbounded operator $A_\ep:D(A_\ep)\subset\mathcal{H}_\ep\rightarrow\mathcal{H}_\ep$ by
\[
A_\ep:=\begin{pmatrix} 0 & 1 & 0 & 0 \\ \Delta-1 & -1 & 0 & 0 \\ 0 & 0 & 0 & \frac{1}{\ep} \\ 0 & -1 & -\frac{1}{\ep} & -\frac{1}{\ep} \end{pmatrix}.
\]
For each $\ep\in(0,1]$, the operator $A_\ep$ with domain $D(A_\ep)$ is an infinitesimal generator of a strongly continuous semigroup of contractions on $\mathcal{H}_\ep$, denoted $e^{A_\ep t}$. According to \cite{Frigeri10}, the $\ep=1$ case follows from \cite[Theorem 2.1]{Beale76}. For each $\ep\in(0,1]$, $A_\ep$ is dissipative because, for all $\zeta=(u,v,\delta,\gamma)\in D(A_\ep)$,
\[
\langle A_\ep\zeta,\zeta \rangle_{\mathcal{H}_\ep} = -\|v\|^2  - \frac{1}{\ep}\|\gamma\|^2_{L^2(\Gamma)} \leq 0.
\]
Also, the Lax--Milgram theorem can be applied to show that the elliptic system, $(I+A_\ep)\zeta=\xi$, admits a unique weak solution $\zeta\in D(A_\ep)$ for any $\xi\in\mathcal{H}_\ep$. Thus, $R(I+A_\ep)=\mathcal{H}_\ep$. 

For each $\ep\in(0,1]$, the map $\mathcal{G}_\ep:\mathcal{H}_\ep\rightarrow\mathcal{H}_\ep$ given by 
\[
\mathcal{G}_\ep(\zeta):=\begin{pmatrix} 0 \\ -f(u) \\ 0 \\ -\left( 1- \frac{1}{\ep} \right) \delta \end{pmatrix}
\]
for all $\zeta=(u,v,\delta,\gamma)\in\mathcal{H}_\ep$ is locally Lipschitz continuous because the map $f:H^1(\Omega)\rightarrow L^2(\Omega)$ is locally Lipschitz continuous. 
Then Problem (A) may be put into the abstract form in $\mathcal{H}_\ep$
\begin{equation}\label{abstract-acoustic-problem}
\left\{
\begin{array}{l} \displaystyle\frac{d\zeta}{dt} =A_\ep\zeta + \mathcal{G}_\ep(\zeta) \\ 
\zeta(0)=\zeta_0 \end{array}
\right.
\end{equation}
where $\zeta=\zeta(t)=(u(t),u_t(t),\delta(t),\delta_t(t))$ and $\zeta_0=(u_0,u_1,\delta_0,\delta_1)\in\mathcal{H}_\ep$, now where $v=u_t$ and $\gamma=\delta_t$ in the sense of distributions.

To obtain the {\em{energy equation}} for Problem (A), multiply (\ref{pde}) by $2u_t$ in $L^2(\Omega)$ and multiply (\ref{abc}) by $2\delta_t$ in $L^2(\Gamma)$, then sum the resulting identities to obtain
\begin{equation}\label{acoustic-energy-3}
\frac{d}{dt} \left\{ \|\zeta\|^2_{\mathcal{H}_\ep} + 2\int_\Omega F(u) dx d\sigma \right\} + 2\|u_t\|^2 + 2\|\delta_t\|^2_{L^2(\Gamma)} = 0,
\end{equation}
where $F(s)=\int_0^s f(\xi) d\xi$. 

\begin{lemma}  \label{adjoint-a}
For each $\varepsilon \in (0,1]$, the adjoint of $A_\ep$, denoted 
$A_\ep^*$, is given by 
\begin{equation*}
A_\ep^*:= -\begin{pmatrix} 0 & 1 & 0 & 0 \\ \Delta-1 & 1 & 0 & 0 \\ 0 & 0 & 0 & \frac{1}{\ep} \\ 0 & -1 & -\frac{1}{\ep} & \frac{1}{\ep} \end{pmatrix},
\end{equation*}
with domain 
\begin{equation*}
D(A_\ep^*):=\{(\chi,\psi,\phi,\xi)\in D(\Delta) \times H^1(\Omega) \times L^2(\Gamma) \times L^2(\Gamma) : \partial _{\mathbf{n}}\chi = - \xi \ \text{on} \ \Gamma \}.
\end{equation*}
\end{lemma}

\begin{proof}
The proof is a calculation similar to, e.g., \cite[Lemma 3.1]{Ball04}.
\end{proof}

Again, the definition of weak solution is from \cite{Ball77}.

\begin{definition} 
Let $T>0$. 
A map $\zeta\in C([0,T];\mathcal{H}_\ep)$ is a {\em{weak solution}} of (\ref{abstract-acoustic-problem}) on $[0,T]$ if for each $\xi\in D(A^*_\ep)$ the map $t \mapsto \langle \zeta(t),\xi \rangle_{\mathcal{H}_\ep}$ is absolutely continuous on $[0,T]$ and satisfies, for almost all $t\in[0,T]$,
\begin{equation}  \label{abs-2}
\frac{d}{dt}\langle \zeta(t),\xi \rangle_{\mathcal{H}_\ep} = \langle \zeta(t),A^*_\ep\xi \rangle_{\mathcal{H}_\ep} + \langle \mathcal{G}_\ep(\zeta(t)),\xi \rangle_{\mathcal{H}_\ep}.
\end{equation}
The map $\zeta$ is a weak solution on $[0,\infty)$ (i.e. a {\em{global weak solution}}) if it is a weak solution on $[0,T]$ for all $T>0$. 
\end{definition}

Following \cite{Ball04}, we provide the equivalent notion of a mild solution.

\begin{definition}
Let $T>0$. 
A function $\zeta:[0,T]\rightarrow\mathcal{H}_\ep$ is a weak/mild solution of (\ref{abstract-acoustic-problem}) on $[0,T]$ if and only if $\mathcal{G}_\ep(\zeta(\cdot))\in L^1(0,T;\mathcal{H}_\ep)$ and $\zeta$ satisfies the variation of constants formula, for all $t\in[0,T],$
\[
\zeta(t)=e^{A_\ep t}\zeta_0 + \int_0^t e^{A_\ep(t-s)}\mathcal{G}_\ep(\zeta(s)) ds.
\]
\end{definition}

Again, our notion of weak solution is equivalent to the standard concept of a weak (distributional) solution to Problem (A).
Indeed, since $f:H^{1}\left( \Omega \right) \rightarrow L^{2}\left( \Omega \right) $ is sequentially weakly continuous and continuous and $(\zeta_{t},\theta)\in C^{1}([0,T])$ for all $\theta \in D(A^*)$, and (\ref{abs-2}) is satisfied. 

\begin{definition}  \label{aweak}
A function $\zeta=(u,u_{t},\delta,\delta_t):[0,T]\rightarrow \mathcal{H}_\ep$ is a weak solution of (\ref{abstract-a}) (and, thus of (\ref{pde}), (\ref{ic}), (\ref{abc}) and (\ref{aic})) on $[0,T],$ if, for almost all $t\in \left[ 0,T\right],$
\begin{equation*}
\zeta =(u,u_{t},\delta,\delta_t)\in C([0,T];\mathcal{H}_\ep),
\end{equation*}
and, for each $\psi \in H^{1}(\Omega),$ $\langle u_{t},\psi \rangle \in C^{1}([0,T]) $ with
\begin{equation*}
\frac{d}{dt}\langle u_{t}(t),\psi\rangle + \langle u_{t}(t),\psi\rangle + \langle u(t),\psi \rangle_1 = -\langle f(u(t)),\psi \rangle - \langle \delta_t(t),\psi\rangle_{L^{2}(\Gamma)},
\end{equation*}
and, for each $\phi \in L^{2}(\Gamma),$ $\langle \delta_{t},\phi \rangle \in C^{1}([0,T]) $ with
\begin{align*}
\frac{d}{dt} \ep \langle \delta_{t}(t),\phi\rangle_{L^{2}(\Gamma)} + \langle \delta_{t}(t),\phi \rangle_{L^{2}(\Gamma)} + \langle \delta(t),\phi\rangle _{L^{2}(\Gamma)} = - \langle u_t(t),\phi\rangle_{L^{2}(\Gamma)}.
\end{align*}
\end{definition}

Observe, on the right-hand side of the last equation above, the derivative of $u$, with respect to $t$, holds in the distribution sense since the term $u_t$ does not possess sufficient regularity to conclude that its trace is in $L^2(\Gamma)$.
Also, recall from the previous section that $f:H^1(\Omega)\rightarrow L^2(\Omega)$ is sequentially weakly continuous and continuous.
Recall that, by \cite[Proposition 3.4]{Ball04} and Lemma \ref{adjoint-a}, $\langle \zeta,\xi \rangle_{\mathcal{H}_\varepsilon}\in C([0,T])$ for all $\xi\in D(A^*)$.

The definition of strong solution follows.
First, for each $\ep\in(0,1]$, define the space, 
\[
\mathcal{D}_\ep:=\left\{ (u,v,\delta,\gamma)\in H^2(\Omega)\times H^1(\Omega)\times H^{1/2}(\Gamma)\times H^{1/2}(\Gamma):\partial_{\bf{n}}u = \gamma ~\text{on} ~\Gamma \right\},
\]
and let $\mathcal{D}_\ep$ be equipped with the $\ep$-weighted norm whose square is given by, for all $\zeta=(u,v,\delta,\gamma)\in\mathcal{D}_\ep$, 
\[
\|\zeta\|^2_{\mathcal{D}_\ep}:=\|u\|^2_2+\|v\|^2_1+\|\delta\|^2_{H^{1/2}(\Gamma)}+ \ep \|\gamma\|^2_{H^{1/2}(\Gamma)}.
\]

\begin{definition}  \label{d:acoustic-strong}
Let $\zeta_0 = \left(u_0,u_1,\delta_0,\delta_1\right) \in \mathcal{D}_\ep$: that is, let $\zeta_0\in H^{2}(\Omega )\times H^{1}(\Omega )\times H^{1/2}(\Gamma)\times H^{1/2}(\Gamma)$ be such that the compatibility condition 
\begin{equation*}
\partial _{\bf{n}}u_{0} = \delta_1 \ \text{on} \ \Gamma
\end{equation*}
is satisfied. 
A function $\zeta (t) =(u(t),u_t(t),\delta(t),\delta_t(t))$ is called a (global) strong solution if it is a (global) weak solution in the sense of Definition \ref{aweak} and if it satisfies the following regularity properties:
\begin{equation}  \label{acoustic-regularity-property}
\begin{array}{l}
\zeta \in L^{\infty }([0,\infty);\mathcal{D}_\varepsilon)\quad\text{and}\quad\partial_t\zeta\in L^{\infty }([0,\infty);\mathcal{H}_\varepsilon).
\end{array}
\end{equation}
Therefore, $\zeta(t) = (u(t), u_t(t),\delta(t),\delta_t(t)) $ satisfies the equations (\ref{pde})-(\ref{aic}) almost everywhere; i.e., is a strong solution.
\end{definition}

The first main result in this section is due to \cite[Theorem 1]{Frigeri10}. 

\begin{theorem}  \label{t:ws-a}
Assume (\ref{reg-assf-2}) and (\ref{assf-2}) hold.
Let $\zeta_0\in\mathcal{H}_\ep$. 
For each $\ep\in(0,1]$, there exists a unique global weak solution $\zeta\in C([0,\infty);\mathcal{H}_\ep)$ to (\ref{abstract-acoustic-problem}). 
For each weak solution, the map 
\begin{equation}\label{acoustic-C1-map}
t \mapsto \|\zeta(t)\|^2_{\mathcal{H}_\ep} + 2\int_\Omega F(u(t)) dx 
\end{equation}
is $C^1([0,\infty))$ and the energy equation (\ref{acoustic-energy-3}) holds (in the sense of distributions). 
Moreover, for all $\zeta_0, \xi_0\in\mathcal{H}_\ep$, there exists a positive constant $\nu_1>0$, depending on $\|\zeta_0\|_{\mathcal{H}_\ep}$ and $\|\xi_0\|_{\mathcal{H}_\ep}$, such that for all $t\geq 0$,
\begin{equation}  \label{cd-a}
\|\zeta(t)-\xi(t)\|_{\mathcal{H}_\ep} \leq e^{\nu_1 t} \|\zeta_0 - \xi_0\|_{\mathcal{H}_\ep}.
\end{equation}
Furthermore, when (\ref{reg-assf-3}) holds and $\zeta_0\in\mathcal{D}_\ep$, then there exists a unique global strong solution $\zeta\in C([0,\infty);\mathcal{D}_\ep)$ to (\ref{abstract-acoustic-problem}). 
\end{theorem}

\begin{proof}
We only report the first part of the proof. 
Following \cite[Proof of Theorem 1]{Frigeri10}: The operator $A_\ep$ is the generator of a $C^0$-semigroup of contractions in $\mathcal{H}_\ep$. 
This follows from \cite{Beale76} and the Lumer-Phillips Theorem. 
Also, recall the functional $\mathcal{G}_\ep:\mathcal{H}_\ep\rightarrow\mathcal{H}_\ep$ is locally Lipschitz continuous. 
So there is $T^*>0$ and a maximal weak solution $\zeta\in C([0,T^*),\mathcal{H}_\ep)$ (cf. e.g. \cite{Zheng04}).
To show $T^*=+\infty$, observe that integrating the energy identity (\ref{acoustic-energy-3}) over $(0,t)$ yields, for all $t\in[0,T^*),$
\begin{align}
& \|\zeta(t)\|^2_{\mathcal{H}_\ep} + 2\int_\Omega F(u(t)) dx + \int_0^t \left( 2\|u_\tau(\tau)\|^2d\tau + 2\ep\|\delta_\tau(\tau)\|^2_{L^2(\Gamma)} \right) d\tau   \notag \\
&= \|\zeta_0\|^2_{\mathcal{H}_\ep} + 2\int_\Omega F(u_0) dx.    \label{fit-1} 
\end{align}
Applying (\ref{consf-2}) to (\ref{fit-1}), we find that, for all $t\in[0,T^*)$,
\begin{equation*}
\|\zeta(t)\|_{\mathcal{H}_\ep}\le C(\|\zeta_0\|_{\mathcal{H}_\ep}),
\end{equation*}
with some $C>0$ independent of $\ep$ and $t$; which of course means $T^*=+\infty.$
Moreover, we know that when $\zeta_0\in\mathcal{H}_\ep$ is such that $\|\zeta_0\|_{\mathcal{H}_\ep}\le R$ for all $\ep\in(0,1],$ then there holds the uniform bound, for all $t\ge0,$
\begin{equation}  \label{acoustic-bound}
\|\zeta(t)\|_{\mathcal{H}_\ep}\le Q(R).
\end{equation}

The remainder of the proof follows as in \cite[Theorem 1]{Frigeri10}.
\end{proof}

As above, we formalize the dynamical system associated with Problem (A).

\begin{corollary}  \label{sf-a}
Let the assumptions of Theorem \ref{t:ws-a} be satisfied.
Let $\zeta_0=(u_0,u_1,\delta_0,\delta_1)\in\mathcal{H}_\ep$ and let $u$ and $\delta$ be the unique solution of Problem (A). 
For each $\ep\in(0,1],$ the family of maps $S_\ep=(S_\ep(t))_{t\geq 0}$ defined by 
\begin{align}
&S_\ep(t)\zeta_0(x):=  \notag \\
&(u(t,x,u_0,u_1,\delta_0,\delta_1),u_t(t,x,u_0,u_1,\delta_0,\delta_1),\delta(t,x,u_0,u_1,\delta_0,\delta_1),\delta_t(t,x,u_0,u_1,\delta_0,\delta_1))  \notag
\end{align}
is the semiflow generated by Problem (A). 
The operators $S_\ep(t)$ satisfy
\begin{enumerate}
	\item $S_\ep(t+s)=S_\ep(t)S_\ep(s)$ for all $t,s\geq 0$.
	\item $S_\ep(0)=I_{\mathcal{H}_\ep}$ (the identity on $\mathcal{H}_\ep$)
	\item $S_\ep(t)\zeta_0\rightarrow S_\ep(t_0)\zeta_0$ for every $\zeta_0\in\mathcal{H}_\ep$ when $t\rightarrow t_0$.
\end{enumerate}

Additionally, each mapping $S_\ep(t):\mathcal{H}_\ep\rightarrow\mathcal{H}_\ep$ is Lipschitz continuous, uniformly in $t$ on compact intervals; i.e., for all $\zeta_0, \chi_0\in\mathcal{H}_\ep$, and for each $T\geq 0$, and for all $t\in[0,T]$,
\begin{equation}\label{sf-a-lc}
\|S_\ep(t)\zeta_0-S_\ep(t)\chi_0\|_{\mathcal{H}_\ep} \leq e^{\nu_1 T}\|\zeta_0-\chi_0\|_{\mathcal{H}_\ep}.
\end{equation}
\end{corollary}

\begin{proof}
The proof is not much different from the proof of Corollary \ref{sf-r} above.
The Lipschitz continuity property follows from (\ref{cd-a}).
\end{proof}

The dynamical system $(S_\ep(t),\mathcal{H}_\ep)$ is shown to admit a positively invariant, bounded absorbing set in $\mathcal{H}_\ep$. 
The argument follows \cite[Theorem 2]{Frigeri10}.

\begin{lemma}  \label{t:a-abs-set}
Assume (\ref{reg-assf-2}) and (\ref{assf-2}) hold.
For each $\ep\in(0,1]$, there exists $R_{1}>0$, independent of $\ep$, such that the following holds: for every $R>0$, there exists $t_{1}=t_1(R) \ge 0$, independent of $\ep$ but depending on $R$, so that, for all $\zeta_0\in\mathcal{H}_\ep$ with $\|\zeta_0\|_{\mathcal{H}_\ep} \le R$ for every $\ep\in(0,1]$, and for all $t\geq t_{1}$,
\begin{equation}  \label{acoustic-radius}
\|S_\ep(t)\zeta_0\|_{\mathcal{H}_\ep} \le R_{1}.
\end{equation}
Furthermore, for each $\ep\in(0,1],$ the set 
\begin{equation}  \label{abss-1}
\mathcal{B}_\ep:=\{ \zeta\in \mathcal{H}_\ep : \|\zeta\|_{\mathcal{H}_\ep} \leq R_{1} \}
\end{equation}
is closed, bounded, absorbing, and positively invariant for the dynamical system $(S_\ep,\mathcal{H}_\ep)$.
\end{lemma}

\begin{proof}
The proof for the case $\ep=1$ is given in \cite{Frigeri10}.
We follow the argument to show the perturbation case with arbitrary $\ep\in(0,1]$.
Careful treatment must be given to the constants appearing the argument.
As with Problem (T), the proof relies on Proposition \ref{t:diff-ineq-1}.
For each $\ep\in(0,1]$ and $\zeta=(u,v,\delta,\gamma)\in\mathcal{H}_\ep$, define the functional, $E_\ep:\mathcal{H}_\ep\rightarrow\mathbb{R}$ by
\begin{equation}  \label{acoustic-functional}
E_\ep(\zeta):=\|\zeta\|^2_{\mathcal{H}_\ep}+2\eta\langle u,\delta \rangle_{L^2(\Gamma)}+2\int_\Omega F(u)dx,
\end{equation}
where $\eta>0$ is some constant that will be chosen below.
By (\ref{acoustic-C1-map}), it is not hard to see that $E_\ep(\zeta(\cdot))\in C^1([0,\infty))$.
Following the proof of \cite[Theorem 2]{Frigeri10}, the claim follows once we can show that there holds, for each $\ep\in(0,1]$ and for almost all $t\ge0$,
\begin{equation}  \label{qaz-1}
\frac{d}{dt}E_\ep(\zeta(t)) + m_1\|\zeta(t)\|^2_{\mathcal{H}_\ep}\le M_{1},
\end{equation}
for some suitable positive constants $m_1$ and $M_1$ (both independent of $\ep$).
Indeed, by multiplying (\ref{pde}) with $w:=u_t+\eta u$ in $L^2(\Omega)$, we first obtain, 
\begin{align}
\frac{1}{2}\frac{d}{dt} & \left\{ \|u\|^2_1 + \|w\|^2 + 2\int_\Omega F(u)dx \right\} + \eta\|u\|^2_1 + (1-\eta)\|w\|^2   \notag  \\
& - \eta(1-\eta)\langle u,w\rangle - \langle \delta_t,u_t \rangle_{L^2(\Gamma)} - \eta\langle \delta_t,u \rangle_{L^2(\Gamma)} + \eta\langle f(u),u \rangle = 0.  \notag
\end{align}
Multiplying (\ref{abc}) with $\theta:=\delta_t+\eta\delta$ in $L^2(\Gamma)$, where $\eta>0$ is yet to be determined, yields
\begin{align}
\frac{1}{2}\frac{d}{dt} & \left\{ \|\delta\|^2_{L^2(\Gamma)} + \ep\|\theta\|^2_{L^2(\Gamma)} \right\} + \eta \|\delta\|^2_{L^2(\Gamma)} + (1-\ep\eta)\|\theta\|^2_{L^2(\Gamma)}  \notag  \\ 
& - \eta(1-\ep\eta) \langle \delta,\theta \rangle_{L^2(\Gamma)} = -\langle u_t,\delta_t \rangle_{L^2(\Gamma)} - \eta\langle u_t,\delta \rangle_{L^2(\Gamma)}.  \notag
\end{align}
Summing the above identities, with
\begin{equation}  \label{simple-id}
\frac{1}{2}\frac{d}{dt} \left\{ 2\eta\langle u,\delta \rangle_{L^2(\Gamma)} \right\} = \eta\langle u_t,\delta\rangle_{L^2(\Gamma)} + \eta\langle u,\delta_t\rangle_{L^2(\Gamma)},
\end{equation}
we obtain, for almost all $t\geq0,$
\begin{align}
\frac{1}{2}\frac{d}{dt} & \left\{ \|u\|^2_1+\|w\|^2+\|\delta\|^2_{L^2(\Gamma)}+\ep\|\theta\|^2_{L^2(\Gamma)} \right.  \notag \\
& + \left. 2\eta\langle u,\delta\rangle_{L^2(\Gamma)}+2\int_\Omega F(u)dx \right\}  \notag \\
& + \eta\|u\|^2_1+(1-\eta)\|w\|^2-\eta(1-\eta)\langle u,w\rangle  \notag \\
& + \eta\|\delta\|^2_{L^2(\Gamma)} + \left(\frac{1}{\ep}-\eta\right) \ep \|\theta\|^2_{L^2(\Gamma)} - \eta(1-\ep\eta)\langle\delta,\theta\rangle_{L^2(\Gamma)}  \notag  \\
& - 2\eta\langle u,\theta \rangle_{L^2(\Gamma)} + 2\eta^2\langle u,\delta \rangle_{L^2(\Gamma)} + \eta\langle f(u),u\rangle  \notag  \\ 
& = 0.  \label{qaz-2}
\end{align}
With Young's inequality and (\ref{consf-2}), estimating the five products for any $0<\eta<1$ and for each $\ep\in(0,1]$,
\begin{align}
-\eta(1-\eta)\langle u,w\rangle \ge -\eta^2\|u\|^2_1 - \frac{1}{4}\|w\|^2,  \label{qaz-3} 
\end{align}
\begin{align}
-\eta(1-\ep\eta)\langle\delta,\theta\rangle_{L^2(\Gamma)} \ge -\frac{\eta}{2} \|\delta\|^2_{L^2(\Gamma)} - \frac{\eta}{2\ep} \ep \|\theta\|^2_{L^2(\Gamma)},  \label{qaz-4}
\end{align}
\begin{align}
-2\eta\langle u,\theta \rangle_{L^2(\Gamma)} \ge - \eta^2\|u\|^2_1 - \frac{\eta C_{\overline{\Omega}}}{\ep} \ep \|\theta\|^2_{L^2(\Gamma)},  \label{fook-1}
\end{align}
\begin{align}
2\eta^2\langle u,\delta \rangle_{L^2(\Gamma)} \ge -\eta^2C_{\overline{\Omega}} \|u\|^2_1 - \frac{\eta^2}{4} \|\delta\|^2_{L^2(\Gamma)},  \label{fook-2}
\end{align}
and
\begin{align}
\eta\langle f(u),u\rangle \ge -\eta(1-\mu_0)\|u\|^2_1 - \eta\kappa_f.  \label{qaz-5}
\end{align}
Recall, $\mu_0$ is due to (\ref{consf-1}), and $C_{\overline{\Omega}}>0$ is the constant due to the trace embedding $H^1(\Omega)\hookrightarrow L^2(\Gamma).$
Together, (\ref{qaz-2})-(\ref{qaz-5}) produce, 
\begin{align}
\frac{1}{2}\frac{d}{dt} & \left\{ \|u\|^2_1+\|w\|^2+\|\delta\|^2_{L^2(\Gamma)}+\ep\|\theta\|^2_{L^2(\Gamma)} \right.  \notag  \\
& + \left. 2\eta\langle u,\delta\rangle_{L^2(\Gamma)}+2\int_\Omega F(u)dx \right\}  \notag \\
& + \eta \left( \mu_0 - \eta\left( 2+C_{\overline{\Omega}} \right) \right)\|u\|^2_1 + \left( \frac{3}{4} - \eta \right)\|w\|^2  \notag  \\
& + \frac{\eta}{2} \left( 1 - \frac{\eta}{2} \right) \|\delta\|^2_{L^2(\Gamma)} + \frac{1}{\ep}\left( 1 - \eta\left( \ep + \frac{1}{2} + C_{\overline{\Omega}} \right) \right) \ep\|\theta\|^2_{L^2(\Gamma)}  \notag  \\
& \le \eta \kappa_f.  \notag  
\end{align}
Define 
\begin{equation}  \label{ep-dep}
\eta_{1}:=\min\left\{ \frac{\mu_0}{2+C_{\overline{\Omega}}},\frac{3}{4},\frac{2}{3+2C_{\overline{\Omega}}} \right\} > 0.
\end{equation}
Then choose $\eta=\eta_1$.
With this, also set,
\begin{align}
m_* := \min \left\{ \mu_0 - \eta_1\left( 2+C_{\overline{\Omega}} \right), \frac{3}{4} - \eta_1, \frac{\eta_1}{2} \left( 1 - \frac{\eta_1}{2} \right),1 - \eta_1\left( \frac{3}{2} + C_{\overline{\Omega}} \right) \right\}.   \label{ep-dep-2}
\end{align}
Note that both $\eta_1$ and $m_*$ are independent of $\ep.$ 
Moreover, we can write 
\begin{align}
\frac{d}{dt} & E_\ep(\zeta) + 2m_* \|\zeta\|^2_{\mathcal{H}_\ep} \le 2\eta_1 \kappa_f,  \label{qaz-9}  
\end{align}
and the estimate (\ref{qaz-9}) can be written in the form (\ref{qaz-1}) with
\begin{equation}  \label{ep-dep-1}
m_{1} := 2m_* \quad \text{and} \quad M_{1} := 2\eta_1 \kappa_f.
\end{equation}

We now show the existence of the absorbing set $\mathcal{B}_\ep$ in (\ref{abss-1}) now follows from Proposition \ref{t:diff-ineq-1}.
First, with Young's inequality,
\begin{align}
2\eta |\langle u,\delta \rangle_{L^2(\Gamma)}| & \le 2\eta_1^2 C_{\overline{\Omega}} \|u\|^2_{1} + \frac{1}{2} \|\delta\|^2_{L^2(\Gamma)},  \label{wer-1}
\end{align}
where $C_{\overline{\Omega}}$ is due to the trace embedding, $H^1(\Omega)\hookrightarrow L^2(\Gamma)$.
We may update the smallness of $\eta$ (that is, $\eta_1$ and hence $m_1$) so that, with (\ref{consf-2}), we are able to find constants $C_{1},C_2>0$, both independent of $\ep$, so that the functional $E_\ep(\zeta)$ from (\ref{acoustic-functional}) satisfies, for each $\ep\in(0,1],$ and for each $\zeta=(u,v,\delta,\gamma)\in\mathcal{H}_\ep,$
\begin{equation}  \label{pre-ab}
C_{1}\|\zeta\|^2_{\mathcal{H}_\ep} - \kappa_f \le E_\ep(\zeta) \le C_{2}\|\zeta\|_{\mathcal{H}_\ep}(1+\|\zeta\|^3_{\mathcal{H}_\ep}).
\end{equation}
Thus, for all $\zeta_0\in \mathcal{H}_\ep$ with $\|\zeta_0\|_{\mathcal{H}_\ep} \le R$, for some $R>0$, the upper-bound in (\ref{pre-ab}) is then 
\begin{equation}  \label{zupper-1}
E_\ep(\zeta_0) \le C_{2} R(1+R^3).
\end{equation}
We have established that, for all $R>0$, there exists $\widetilde R>0$, independent of $\ep$, such that, for all $\zeta_0\in\mathcal{H}_\ep$ with $\|\zeta_0\|_{\mathcal{H}_\ep} \le R$, then $E_\ep(\zeta_0)\le \widetilde R.$
From the lower-bound in (\ref{pre-ab}), we find 
\begin{equation}
\sup_{t\ge0}E_\ep(\zeta(t)) \ge \kappa_f.
\end{equation}
By Proposition \ref{t:diff-ineq-1}, for each $\ep\in(0,1]$ and for all $\zeta_0\in \mathcal{H}_\ep$ with $\|\zeta_0\|_{\mathcal{H}_\ep} \le R$, there exists $t_{1}>0$, independent of $\ep$, $\kappa_f$, and $R$, such that, for all $t\ge t_{1}$ and for all $\ep\in(0,1],$
\begin{align}
E_\ep(S_\ep(t)\zeta_0) & \le \sup_{\zeta\in \mathcal{H}_\ep} \left\{ E_\ep(\zeta) : m_{1} \|\zeta\|^2_{\mathcal{H}_\ep} \le M_{1}+\iota \right\}  \notag \\ 
& = \sup_{\zeta\in \mathcal{H}_\ep} \left\{ E_\ep(\zeta) : \|\zeta\|^2_{\mathcal{H}_\ep} \le \frac{1}{m_1} \left( \eta_1\kappa_f + \iota \right) \right\},  \label{wer-2}
\end{align}
for any $\iota>0.$
Thus, using $(\ref{pre-ab})$ we find that for all $\zeta_0\in \mathcal{H}_\ep$ with $\|\zeta_0\|_{\mathcal{H}_\ep} \le R$, for $R>0$, there is $R_{1}>0$ and $t_{1}>0$, both depending on both $R$, but independent of $\ep$, such that for all $t\geq t_{1}$, (\ref{acoustic-radius}) holds.
By definition, the set $\mathcal{B}_\ep$ in (\ref{abss-1}) is closed and bounded in $\mathcal{H}_\ep$. 
We have also shown that $\mathcal{B}_\ep$ is absorbing in the following sense: for any nonempty bounded subset $B$ of $\mathcal{H}_\ep$, there is a $t_{1}\geq 0$, depending on $B$, in which for all $t\geq t_{1}$, $S_\ep(t)B\subseteq \mathcal{B}_\ep$. 
Consequently, since $\mathcal{B}_\ep$ is bounded, $\mathcal{B}_\ep$ is also positively invariant under the semiflow $S_\ep$.

By Proposition \ref{t:diff-ineq-1}, $t_{1}=t_{1}(\iota)$, for $\iota>0$, depends on the parameters of (\ref{consf-2}) and of course $R>0$ as, 
\begin{align}
t_{1}(\iota) & := \frac{1}{\iota}\left( C_{2}R(1+R^3) + \kappa_f\right).  \notag
\end{align}
Furthermore, after some calculation, we find the radii of $\mathcal{B}_\ep$ to be given by, for all $\ep\in(0,1]$ and $\iota>0,$
\begin{align}
& R_{1}(\iota) := C^{-1/2}_1 \left( C_{2} \left( \frac{\eta_1\kappa_f}{m_*}+\frac{\iota}{2m_*} \right)^{1/2} \left( 1+ \left( \frac{\eta_1\kappa_f}{m_*}+\frac{\iota}{2m_*} \right)^{3/2} + \kappa_f \right) \right)^{1/2}.  \notag
\end{align}
We now fix $\iota=2m_*.$ 
Then
\[
t_{1}:=\frac{1}{2m_*}\left( C_2R(1+R^3)+\kappa_f\right)
\]
(which is independent of $\ep$).
Moreover, 
\begin{align}
R_{1}:= C^{-1/2}_1 \left( C_{2} \left( \frac{\eta_1\kappa_f}{m_*}+1 \right)^{1/2} \left( 1+ \left( \frac{\eta_1\kappa_f}{m_*}+1 \right)^{3/2} + \kappa_f \right) \right)^{1/2}    \label{g-attr-bound}
\end{align}
is also independent of $\ep.$
This concludes the proof.
\end{proof}

The existence of a global attractor leverages the existence of a regular absorbing set, for each $\ep\in(0,1],$ for the Problem (A) (cf. \cite{Babin&Vishik92} or \cite{Temam88}).

\begin{theorem}  \label{optimal-a}
Assume (\ref{reg-assf-2}), (\ref{assf-2}), and (\ref{reg-assf-3}) hold.
For each $\ep\in(0,1]$, there exists a closed and bounded subset $\mathcal{U}_\ep\subset \mathcal{D}_\ep$, such that for every nonempty bounded subset $B\subset \mathcal{H}_\ep$,
\begin{equation}
dist_{\mathcal{H}_\ep}(S_\ep(t)B,\mathcal{U}_\ep) \le Q_\ep(\left\Vert B\right\Vert _{\mathcal{H}_\ep}) e^{-\omega_{2\ep} t},  \label{tran-2}
\end{equation}
for some nonnegative monotonically increasing function $Q_{\ep}(\cdot)$ and for some positive constant $\omega_{2\ep}>0$, both depending on $\ep$ where $Q_\ep(\cdot)\sim\ep^{-1}$ and $\omega_{2\ep}\sim\ep^{-1}.$
\end{theorem} 

Thus, just like in the previous section, we are able to define a family of global attractors in $\mathcal{H}_\ep$ (cf. e.g. (\ref{gfam-1})) that attain optimal regularity; that is, trajectories on the attractor are strong solutions. 

\begin{theorem}  \label{t:gattr-a}
Let the assumptions of Theorem \ref{optimal-a} hold. 
For each $\ep\in(0,1]$, the dynamical system $(S_\ep(t),\mathcal{H}_\ep)$ admits a global attractor $\mathcal{A}_\ep$ in $\mathcal{H}_\ep$. 
The global attractor is invariant under the semiflow $S_\ep$ (both positively and negatively) and attracts all nonempty bounded subsets of $\mathcal{H}_\ep$; precisely, 
\begin{enumerate}
\item For each $t\geq 0$, $S_\ep(t)\mathcal{A}_\ep=\mathcal{A}_\ep$, and 
\item For every nonempty bounded subset $B$ of $\mathcal{H}_\ep$,
\[
\lim_{t\rightarrow\infty}{\rm{dist}}_{\mathcal{H}_\ep}(S_\ep(t)B,\mathcal{A}_\ep):=\lim_{t\rightarrow\infty}\sup_{\varphi\in B}\inf_{\theta\in\mathcal{A}_\ep}\|S_\ep(t)\varphi-\theta\|_{\mathcal{H}_\ep}=0.
\]
\end{enumerate}
The global attractor is unique and given by 
\[
\mathcal{A}_\ep=\omega(\mathcal{B}_\ep):=\bigcap_{s\geq t}{\overline{\bigcup_{t\geq 0} S_\ep(t)\mathcal{B}_\ep}}^{\mathcal{H}_\ep}.
\]
Furthermore, $\mathcal{A}_\ep$ is the maximal compact invariant subset in $\mathcal{H}_\ep$.
\end{theorem}

Moreover, we immediately have the following

\begin{corollary}
Let the assumptions of Theorem \ref{t:gattr-a} hold. 
For each $\ep\in(0,1]$, the global attractor $\mathcal{A}_\ep$ admitted by the semiflow $S_\ep$ satisfies 
\begin{equation*}
\mathcal{A}_\ep\subset\mathcal{U}_\ep.
\end{equation*}
Consequently, for each $\ep\in(0,1]$, the global attractor $\mathcal{A}_\ep$ is bounded in $\mathcal{D}_\ep$ and consists only of strong solutions. 
\end{corollary}

The proof of Theorem \ref{optimal-a} proceeds along the usual lines; whereby decomposing the semiflow $S_\ep$ into two parts, one which decays (exponentially) to zero, and one part which is precompact. 
Again, the results follow the presentation in \cite{Frigeri10} with minor modifications to include the perturbation parameter $\ep\in(0,1].$
Define 
\begin{equation}  \label{beta}
\psi(s):=f(s)+\beta s
\end{equation}
for some constant $\beta \ge \vartheta$ to be determined later (observe, $\psi'(s)\ge 0$ thanks to assumption (\ref{reg-assf-3})). 
Set $\Psi (s):=\int_{0}^{s}\psi (\sigma )d\sigma$. 
Let $\zeta_0=(u_0,u_1,\delta_0,\delta_1)\in\mathcal{H}_\ep$. 
Let $\zeta(t)=(u(t),u_t(t),\delta(t),\delta_t(t))$ denote the corresponding global solution of Problem (A) on $[0,\infty)$ with the initial data $\zeta_0$. 
We then decompose Problem (A) into the following systems of equations.
For all $t\ge0$, set 
\begin{align}
\zeta(t) & = (u(t),u_t(t),\delta(t),\delta_t(t))  \notag \\ 
& = \left( v(t),v_t(t),\gamma(t),\gamma_t(t) \right) + \left( w(t),w_t(t),\theta(t),\theta_t(t) \right)  \notag \\
& =: \xi(t) + \chi(t)  \notag 
\end{align}
Then $\xi$ and $\chi$ satisfy the IBVPs,
\begin{equation}
\left\{ 
\begin{array}{ll}  \label{pde-v}
v_{tt} + v_{t} - \Delta v + v + \psi (u) - \psi (w) = 0 & \text{in}\quad(0,\infty
)\times \Omega, \\ 
\ep \gamma_{tt} + \gamma_t + \gamma = -v_t & \text{on}\quad(0,\infty)\times \Gamma, \\ 
\gamma_t = \partial_{\bf{n}}v & \text{on}\quad(0,\infty )\times \Gamma, \\ 
\xi(0) = \zeta_0 & \text{at}\quad\{0\}\times{\overline{\Omega}}, \\ 
\end{array}
\right.\end{equation}
and, respectively, 
\begin{equation}
\left\{ 
\begin{array}{ll}  \label{pde-w}
w_{tt} + w_{t} - \Delta w + w + \psi (w) = \beta u & \text{in}\quad(0,\infty)\times \Omega, \\ 
\ep \theta_{tt} + \theta_t + \theta = -w_t & \text{on}\quad(0,\infty)\times \Gamma, \\ 
\theta_{t} = \partial_{\bf{n}}w & \text{on}\quad(0,\infty)\times \Gamma, \\ 
\chi(0) = {\bf{0}} & \text{at}\quad\{0\}\times{\overline{\Omega}}.
\end{array} \right.
\end{equation}
In view of Lemmas \ref{t:uniform-bound-w} and \ref{t:uniform-decay} below, we define the one-parameter family of maps, $K_\ep(t):\mathcal{H}_\ep\rightarrow \mathcal{H}_\ep$, by 
\begin{equation*}
K_\ep(t)\zeta_{0} := \chi(t),
\end{equation*}
where $\chi(t)$ is a solution of (\ref{pde-w}). 
With such $\chi(t)$, we may define a second function $\xi(t)$ for all $t\ge 0$ as the solution of (\ref{pde-v}). 
Through the dependence of $\xi$ on $\chi$ and $\zeta_{0}$, the solution of (\ref{pde-v}) defines a one-parameter family of maps, $Z_\ep(t):\mathcal{H}_{\varepsilon}\rightarrow \mathcal{H}_\ep$, defined by 
\begin{equation*}
Z_\ep(t)\zeta_{0} := \xi(t).
\end{equation*}
Notice that if $\xi$ and $\chi$ are solutions to (\ref{pde-v}) and (\ref{pde-w}), respectively, then the function $\zeta := \xi+\chi$ is a solution to the
original Problem (A), for each $\ep\in(0,1]$.

The first lemma shows that the operators $K_\ep$ are bounded in $\mathcal{H}_\ep$, uniformly with respect to $\varepsilon$. 
The result largely follows from the existence of a bounded absorbing set $\mathcal{B}_\ep$ in $\mathcal{H}_\ep$ for $S_\ep$ (recall (\ref{acoustic-bound})).

\begin{lemma}  \label{t:uniform-bound-w} 
Let the assumptions of Theorem \ref{t:gattr-a} hold. 
For each $\varepsilon \in (0,1]$ and $\zeta _{0}=(u_{0},u_{1},\delta_0,\delta_1) \in \mathcal{H}_\ep$, there exists a unique global weak solution $\chi = (w, w_{t}, \theta, \theta_t) \in C([0,\infty );\mathcal{H}_\ep)$ to problem (\ref{pde-w}) satisfying 
\begin{equation}  \label{bnd-2}
\theta_{t}\in L_{loc}^{2}([0,\infty )\times \Gamma ).
\end{equation}
Moreover, for all $\zeta_{0}\in \mathcal{H}_\ep$ with $\left\Vert \zeta_{0}\right\Vert _{\mathcal{H}_\ep}\leq R$ for
all $\varepsilon \in (0,1]$, there holds for all $t\geq 0$, 
\begin{equation}  \label{uniform-bound-w}
\Vert K_\ep(t)\zeta _{0}\Vert _{\mathcal{H}_\ep} \le Q(R).
\end{equation}
\end{lemma}

\begin{lemma}  \label{t:Gronwall-bound} 
Let the assumptions of Theorem \ref{t:gattr-a} hold. 
For each $\ep\in(0,1]$ and for all $\eta >0$, there is a function $Q_\eta(\cdot) \sim \eta^{-1}$, such that for every $0\leq s\leq t$, $\zeta_{0} = (u_{0},u_{1},\delta_0,\delta_1) \in \mathcal{B}_{\varepsilon }$, and $\ep\in(0,1]$,
\begin{equation}  \label{Gronwall-bound-0}
\int_{s}^{t} \left( \Vert u_{t}(\tau )\Vert ^{2} + \Vert w_{t}(\tau )\Vert
^{2} + \|\theta(\tau)\|^2_{L^2(\Gamma)} \right) d\tau \leq \frac{\eta }{2}(t - s) + Q_{\eta}(R),
\end{equation}
where $R>0$ is such that $\left\Vert \zeta _{0}\right\Vert _{\mathcal{H}_{\varepsilon }}\leq R,$ for all $\varepsilon\in(0,1]$.
\end{lemma}

The next result shows that the operators $Z_\ep$ are uniformly decaying to zero in $\mathcal{H}_\ep$, for each $\ep\in(0,1]$.

\begin{lemma}  \label{t:uniform-decay}  
Let the assumptions of Theorem \ref{t:gattr-a} hold. 
For each $\varepsilon \in (0,1]$ and $\zeta_{0}=(u_{0},u_{1},\delta_0, \delta_1) \in \mathcal{H}_{\varepsilon }$, there exists a unique global weak solution $\xi=(v,v_t,\gamma,\gamma_t)\in C([0,\infty );\mathcal{H}_{\varepsilon })$ to problem (\ref{pde-v}) satisfying 
\begin{equation}  \label{bounded-boundary-3}
\gamma_{t}\in L_{loc}^{2}([0,\infty )\times \Gamma ).
\end{equation}
Moreover, for all $\zeta_0\in\mathcal{D}_\varepsilon$ with $\left\Vert\zeta_{0}\right\Vert _{\mathcal{H}_{\varepsilon }}\leq R$ for all $\varepsilon \in (0,1]$, there is a constant $\omega_{3\ep} > 0$, depending on $\ep$ as $\omega_{3\ep}\sim\ep$, and there is a positive monotonically increasing function $Q_\ep(\cdot)\sim\ep^{-1}$, such that, for all $t\geq 0$, 
\begin{equation}  \label{uniform-decay}
\Vert Z_{\varepsilon }(t)\zeta_{0}\Vert _{\mathcal{H}_{\varepsilon }} \le
Q_\ep(R)e^{-\omega_{3\ep} t}.
\end{equation}
\end{lemma}

The following lemma establishes the precompactness of the operators $K_{\varepsilon }$, for each $\ep\in(0,1].$

\begin{lemma}  \label{compact-a}
Let the assumptions of Theorem \ref{t:gattr-a} hold. 
For all $\ep\in(0,1]$, and for each $R>0$ and $\zeta\in\mathcal{D}_\ep$ such that $\|\zeta\|_{\mathcal{D}_\ep}\le R$ for all $\ep\in(0,1]$, there exist constants $\omega_{4\ep},R_{2\ep}>0$, both depending on $\ep$, with $\omega_{4\ep}\sim\ep$ and $R_{2\ep}\sim\ep^{-4}$, in which, for all $t\ge0$, there holds
\begin{equation}  \label{a-exp-attr-1}
\|K_\ep(t)\zeta_0\|^2_{\mathcal{D}_\ep} \le Q(R)e^{-\omega_{4\ep}t/2} + R_{2\ep}.
\end{equation}
\end{lemma}

\begin{remark}
Thus, so far, we have shown that the semiflows associated with Problem (T) and Problem (A) admit global attractors that are bounded, independent of the perturbation parameter $\ep$, in the space $\mathcal{H}_\ep$. 
Although each attractor is also bounded in $\mathcal{D}_\ep$, the bound for those $\mathcal{A}_\ep\subset\mathcal{U}_\ep$, $\ep\in(0,1]$, {\em{does}} depend on $\ep$ in a crucial way; precisely, $R_{2\ep}\rightarrow+\infty$ as $\ep\rightarrow0.$
\end{remark}

We now turn to the existence of exponential attractors for each $\varepsilon\in(0,1]$ for Problem (A). 
By \cite[Theorem 5]{Frigeri10}, we already know that Problem (A) with $\ep=1$ admits an exponential attractor described by Theorem \ref{eaa} below.
We aim to show that the results for the perturbed case follow from \cite{Frigeri10} after suitable modifications to include $\varepsilon\in(0,1]$ appearing in the boundary condition (\ref{abc}). 

\begin{theorem}  \label{eaa} 
Assume (\ref{reg-assf-2}), (\ref{assf-2}), and (\ref{reg-assf-3}) hold.
For each $\varepsilon \in (0,1]$, the dynamical system $(S_\ep,\mathcal{H}_\ep)$ associated with Problem (A) admits an exponential attractor $\mathcal{M}_{\varepsilon }$ compact in $\mathcal{H}_{\varepsilon},$ and bounded in $\mathcal{D}_\ep$. 
Moreover, for each $\ep\in(0,1]$ fixed, there hold:

(i) For each $t\geq 0$, $S_{\varepsilon }(t)\mathcal{M}_{\varepsilon}\subseteq \mathcal{M}_{\varepsilon }$.

(ii) The fractal dimension of $\mathcal{M}_{\varepsilon }$ with respect to the metric $\mathcal{H}_{\varepsilon }$ is finite, namely,
\begin{equation*}
\dim_{F}\left( \mathcal{M}_{\varepsilon },\mathcal{H}_{\varepsilon }\right) \leq C_{\varepsilon }<\infty,
\end{equation*}
for some positive constant $C=C_\ep$ depending on $\varepsilon$.

(iii) There exists a nonnegative monotonically increasing function $Q_{\ep}$ and a positive constant $\varrho_{1\ep}>0$, both depending on $\ep$, such that, for all $t\geq 0$, 
\begin{equation*}
dist_{\mathcal{H}_{\varepsilon }}(S_{\varepsilon }(t)B,\mathcal{M}_{\varepsilon })\leq Q_{\ep}(\Vert B\Vert _{\mathcal{H}_{\varepsilon }})e^{-\varrho_{1\ep} t},
\end{equation*}
for every nonempty bounded subset $B$ of $\mathcal{H}_{\varepsilon }$.
\end{theorem}

As with Problem (T) above, the proof of Theorem \ref{eaa} will follow from the application of an abstract proposition reported specifically for our current case below (see, e.g., \cite[Proposition 1]{EMZ00}, \cite{FGMZ04}, \cite{GGMP05}).

\begin{proposition}  \label{abstract-a}
Let $(S_{\varepsilon},\mathcal{H}_{\varepsilon}) $ be a dynamical system for each $\varepsilon >0$. 
Assume the following hypotheses hold for each $\ep\in(0,1]$:

\begin{enumerate}

\item[(H1)] There exists a bounded absorbing set $\mathcal{B}_{\varepsilon
}^{1}\subset \mathcal{D}_{\varepsilon }$ which is positively invariant for $S_{\varepsilon }(t).$ 
More precisely, there exists a time $t_{2\ep}>0,$ which \emph{depends} on $\varepsilon >0$, such that
\begin{equation*}
S_{\varepsilon }(t)\mathcal{B}_{\varepsilon }^{1}\subset \mathcal{B}_{\varepsilon }^{1}
\end{equation*}
for all $t\geq t_{2\ep}$ where $\mathcal{B}_{\varepsilon }^{1}$ is endowed with
the topology of $\mathcal{H}_{\varepsilon }.$

\item[(H2)] There is $t^*_\ep \ge t_{2\ep}$ such that the map $S_{\varepsilon
}(t^*)$ admits the decomposition, for each $\varepsilon \in (0,1]$ and
for all $\zeta_{0},\xi_{0}\in \mathcal{B}_{\varepsilon }^{1}$, 
\begin{equation*}
S_{\varepsilon }(t^*)\zeta_{0}-S_{\varepsilon }(t^*)\xi_{0}=L_{\varepsilon }(\zeta_{0},\xi_{0})+R_{\varepsilon }(\zeta_{0},\xi_{0})
\end{equation*}
where, for some constants $\alpha^*\in(0,\frac{1}{2})$ and $\Lambda
^*_\ep = \Lambda^*(\ep,\Omega,t^*)\geq 0$, the following hold:
\begin{equation}\label{dd-l-a}
\Vert L_{\varepsilon }(\zeta_{0},\xi_{0})\Vert _{\mathcal{H}_{\varepsilon }}\leq \alpha^*\Vert \zeta_{0}-\xi_{0}\Vert _{\mathcal{H}_{\varepsilon }}  
\end{equation}
and
\begin{equation}  \label{dd-k-a}
\Vert R_{\varepsilon }(\zeta_{0},\xi_{0})\Vert _{\mathcal{D}_{\varepsilon }}\leq \Lambda^*_\ep \Vert \zeta_{0}-\xi_{0}\Vert _{\mathcal{H}_{\varepsilon }}.
\end{equation}

\item[(H3)] The map
\begin{equation*}
(t,U)\mapsto S_{\varepsilon }(t)\zeta:[t^*_\ep,2t^*_\ep]\times \mathcal{B}_{\varepsilon }^{1}\rightarrow \mathcal{B}_{\varepsilon }^{1}
\end{equation*}
is Lipschitz continuous on $\mathcal{B}_{\varepsilon }^{1}$ in the topology
of $\mathcal{H}_{\varepsilon}$.
\end{enumerate}

Then, $(S_{\varepsilon },\mathcal{H}_{\varepsilon })$ possesses
an exponential attractor $\mathcal{M}_{\varepsilon }$ in $\mathcal{B}_{\varepsilon }^{1}.$
\end{proposition}

\begin{remark}
As in the case for Problem (T), the basin of exponential attraction for the exponential attractors for Problem (A) is indeed the entire phase space thanks to (\ref{a-exp-attr-1}) below.
This can be shown by using the exponential attraction of subsets of $\mathcal{B}^1_\varepsilon$ to $\mathcal{M}_\ep$, and by referring to the transitivity of exponential attraction (see Lemma \ref{t:exp-attr}).
\end{remark}

\begin{corollary}  \label{dim}
Under the assumptions of Theorem \ref{eaa}, for each $\ep\in(0,1]$, the global attractors of Theorem \ref{t:gattr-a} are bounded in $\mathcal{D}_\ep$.

In addition, there holds
\begin{equation*}
\dim_{F}(\mathcal{A}_\varepsilon,\mathcal{H}_\varepsilon)\leq \dim_{F}(\mathcal{M}_\varepsilon,\mathcal{H}_\varepsilon)\le C_\ep,
\end{equation*}
for some constant $C_\ep>0$, depending on $\ep.$
As a consequence, each of the global attractors $\mathcal{A}_\varepsilon$, for $\ep\in(0,1]$, is finite dimensional; however, the dimension of $\mathcal{M}_{\varepsilon }$ is not necessarily uniform with respect to $\ep$.
\end{corollary}

To prove Theorem \ref{eaa}, we apply the abstract result expressed in Proposition \ref{abstract-a}.
As a preliminary step, we make an observation on the energy equation (\ref{acoustic-energy-3}) associated with Problem (A).

We now show that the assumptions (H1)-(H3) hold for $(S_{\varepsilon}(t),\mathcal{H}_{\varepsilon })$, for each $\ep\in(0,1]$. 

\begin{lemma}  \label{t:to-H1} 
Let the assumptions of Theorem \ref{eaa} be satisfied.
Condition (H1) holds for fixed $\varepsilon \in (0,1]$. 
Moreover, for each $\ep\in(0,1]$, and for each $R>0$ and $\zeta\in\mathcal{D}_\ep$ such that $\|\zeta\|_{\mathcal{D}_\ep}\le R$ for all $\ep\in(0,1]$, there exist constants $\omega_{5\ep},R_{3\ep}>0$, both depending on $\ep$, with $\omega_{5\ep}\sim\ep$ and $R_{3\ep}\sim\ep^{-4}$, in which, for all $t\ge0$, there holds
\begin{equation}  \label{a-exp-attr-2}
\|S(t)\zeta_0\|^2_{\mathcal{D}_\ep} \le Q(R)e^{-\omega_{5\ep}t/2} + R_{3\ep}.
\end{equation}
\end{lemma}

\begin{proof} 
We follow the proof of Lemma \ref{compact-a}, but now with $u$ in place of $w$, $\delta$ in place of $\gamma,$ and $\beta=0$ (for more details, see \cite[Lemma 8]{Frigeri10} or \cite[Lemma 3.16]{Gal&Shomberg15}).
Following the same estimates, we deduce that, for some constants $\omega_{5\ep}\sim\ep$ and $R_{3\ep}\sim\ep^{-4}$, for all $t\ge 0,$
\begin{equation}  \label{for-C1-3}
\|\zeta(t)\|^2_{\mathcal{H}_\ep} \le Q(R) e^{-\omega_{5\ep} t/2} + R_{3\ep}.
\end{equation}
This shows (\ref{a-exp-attr-2}).

Furthermore, the existence of the bounded absorbing set now follows. 
Indeed, for each $\ep\in(0,1]$, there is a constant $R_{2\ep} > \widetilde R_{2\ep}$, where $R_{2\ep}\sim\ep^{-4}$, in which the set
\begin{equation}
\mathcal{B}_\ep^1:=\left\{ \zeta\in\mathcal{D}_\ep:\|\zeta\|_{\mathcal{D}_\ep}\le R_{3\ep} \right\}
\end{equation}
is a bounded absorbing set in $\mathcal{D}_\ep$, positively invariant for $S_\ep(t)$: for each $\ep\in(0,1]$ and for each bounded subset $B\subset\mathcal{H}_\ep$, there is $t_{2 \ep}=t_2(B,R_{3\ep})\ge0$, in which, for all $t\ge t_{2 \ep}$,
\[
S_\ep(t)B\subseteq\mathcal{B}_\ep^1,  
\]
with respect to the topology of $\mathcal{H}_\ep$ (obviously, the radius of $\mathcal{B}^1_\ep$ may need to be enlarged after applying the embedding $\mathcal{D}_\ep\hookrightarrow\mathcal{H}_\ep$).
This completes the proof of (H1).
\end{proof}

\begin{remark}  \label{time-2}
The ``time of entry'' of some bounded set $B\subset\mathcal{D}_\ep$ into $\mathcal{B}^1_\ep$ is given by
\[
t_{2\ep}=t_{2}(\|B\|_{\mathcal{H}_\ep},R_{3\ep}) = \max\left\{ \frac{1}{\omega_{5\ep}}\ln\left( \frac{Q\left(\|B\|_{\mathcal{H}_\ep}\right)}{R_{2\ep} - R_{3\ep}} \right),0 \right\}
\]
where $R_{2\ep}$ is the radius of the absorbing set $\mathcal{B}^1_\ep$ in $\mathcal{D}_\ep$.
Furthermore, both $R_{2 \ep}$ and $t_{2 \ep} \rightarrow+\infty$ as $\ep\rightarrow 0$.
\end{remark}

\begin{corollary}
Let the assumptions of Theorem \ref{eaa} be satisfied.
For each $\ep\in(0,1]$, $\zeta_0\in\mathcal{D}_\ep$, there is $P_{\ep}(R)=P(R_{3\ep},R)>0$, where $P_\ep\sim\ep^{-4}$, in which there holds, 
\begin{equation*}  \label{reg-bound}
\liminf_{t\rightarrow\infty}\|S_\ep(t)\zeta_0\|_{\mathcal{D}_\ep}\le P_\ep(R).
\end{equation*}
\end{corollary}

The proofs of the remaining conditions (H2) and (H3) follow directly from \cite{Frigeri10} with only minor modifications. 
However, in light of Remark \ref{time-2}, the results are not uniform in $\ep\in(0,1].$

\begin{lemma}  \label{t:to-H2} 
Let the assumptions of Theorem \ref{eaa} be satisfied.
Conditions (H2) and (H3) hold for each fixed $\varepsilon \in (0,1]$.
\end{lemma}

According to Proposition \ref{abstract-a}, the proof or Theorem \ref{eaa} is complete.
Recalling Corollary \ref{dim}, the dimension of the corresponding global attractor, $\mathcal{A}_{\ep}$, $\ep\in[0,1]$, is insured to possess {\em{finite}} dimension, but again, due to Remark \ref{time-2}, this bound is not necessarily uniform in $\ep$.

\section{The continuity properties of the attractors}  \label{s:cont}

This section contains a new abstract theorem (Theorem \ref{t:robustness}) concerning the upper-semicontinuity of a family of sets. 
The key assumption in the theorem involves a comparison of the semiflow corresponding to the unperturbed problem to the semiflow corresponding to the perturbation problem in the topology of the perturbation problem. 
The unperturbed problem is ``fitted'' into the phase space of the perturbed problem through the use of two maps, a {\em{canonical extension}} and {\em{lift}}. This approach for obtaining an upper-semicontinuous family of sets developed in this section is largely motivated by \cite{GGMP05}. In the setting presented in this paper, where the perturbation is isolated to the boundary condition, the upper-semicontinuity result is obtained for a broad range of families of sets and the overall analysis is much simpler. The result is applied to Problem (T) and Problem (A) in the final part of this section. 

The abstract upper-semicontinuity theorem is developed in this section.

\begin{definition} Given two bounded subsets $A$, $B$ in a Banach space $X$, the {\em Hausdorff (asymmetric) semidistance} between $A$ and $B$, in the topology of $X$, is defined by 
\[
dist_X(A,B):=\sup_{a\in A}\inf_{b\in B}\|a-b\|_X.
\]
\end{definition} 

Suppose $X_0$ is a Banach space with the norm $\|\chi\|_{X_0}$, for all $\chi\in X_0$, and suppose $Y$ is a Banach space with the norm $\|\psi\|_Y$, for all $\psi\in Y$. For $\ep\in(0,1]$, let $X_\ep$ be the one-parameter family of Banach product spaces 
\[
X_\ep=X_0\times Y
\] 
with the $\ep$-weighted norm whose square is given by 
\[
\|(\chi,\psi)\|^2_{X_\ep}=\|\chi\|^2_{X_0} + \ep\|\psi\|^2_{Y}.
\]
For each $\ep\in(0,1]$, let $S_\ep$ be a semiflow on $X_\ep$ and let $S_0$ be a semiflow on $X_0$. Let $\Pi:X_\ep\rightarrow X_0$ be the {\em{projection}} from $X_\ep$ onto $X_0$; for every subset $B_\ep\subset X_\ep$, $\Pi B_\ep=B_0\subset X_0$.

Define the ``lift'' map to map sets $B_0\subset X_0$ to sets in the product $X_\ep$. With the lift map it is possible to measure the semi-distance between sets from $X_0$ with sets in $X_\ep$, using the topology of $X_\ep$. 

\begin{definition}
Given a map $\mathcal{E}:X_0\rightarrow Y$, locally Lipschitz in $X_0$, the map $\mathcal{L}:X_0\rightarrow X_\ep$ defined by $\chi\mapsto (\chi,\mathcal{E}\chi)$ is called a {\em{lift}} map. 
The map $\mathcal{E}$ is called a {\em{canonical extension}}. 
If $B_0$ is a bounded subset of $X_0$, the set $\mathcal{E}B_0\subset Y$ is called the canonical extension of $B_0$ into $Y$, and the set 
\[
\mathcal{L}B_0=\{(\chi,\psi)\in X_\ep : \chi\in B_0, \ \psi\in \mathcal{E}B_0\}
\]
is called the lift of $B_0$ into $X_\ep$.
\end{definition}

What follows is a general description the type of families of sets that will be upper-semicontinuous in $X_\ep$.

Let $W_0$ be a bounded subset of $X_0$, and let $S_0$ be a semiflow on $X_0$. 
For each $\ep\in(0,1]$, let $W_\ep$ be a bounded subset of $X_\ep$, and let $S_\ep$ be a semiflow on $X_\ep$. 
Let $T>0$ and define the sets 
\begin{equation}\label{set-family-0}
\mathcal{U}_0 = \bigcup_{t\in[0,T]} S_0(t)W_0
\end{equation}
and
\begin{equation}\label{set-family-1}
\mathcal{U}_\ep = \bigcup_{t\in[0,T]} S_\ep(t)W_\ep.
\end{equation}
Define the family of sets $(\mathbb{U}_\ep)_{\ep\in[0,1]}$ in $X_\ep$ by
\begin{equation}\label{set-family-2}
\mathbb{U}_\ep=\left\{ \begin{array}{ll} \mathcal{U}_\ep & 0<\ep\leq 1 \\ \mathcal{L}\mathcal{U}_0 & \ep=0. \end{array} \right.
\end{equation}

\begin{remark}
The sets $W_0$ and $W_\ep$ are not assumed to be positively invariant.
\end{remark}

\begin{theorem}\label{t:robustness}
Suppose that the semiflow $S_0$ is Lipschitz continuous on $X_0$, uniformly in $t$ on compact intervals. 
Suppose the lift map $\mathcal{L}$ satisfies the following: for any $T>0$ and $B_\ep\subset X_\ep$ in which there exists $M=M(\|B_\ep\|_{X_\ep})>0$, depending on $B_\ep$, $\rho\in(0,1]$, both independent of $\ep$, such that for all $t\in [0,T]$ and $(\chi,\psi)\in B_\ep$,
\begin{equation}\label{robustness}
\|S_\ep(t)(\chi,\psi)-\mathcal{L}S_0(t)\Pi(\chi,\psi)\|_{X_\ep}\le M\ep^\rho.
\end{equation}
Then the family of sets $(\mathbb{U}_\ep)_{\ep\in[0,1]}$ is upper-semicontinuous in the topology of $X_\ep$; precisely, 
\[
{\rm dist}_{X_\ep}(\mathbb{U}_\ep,\mathbb{U}_0) \le M\ep^\rho.
\]
\end{theorem}

\begin{proof}
To begin,
\[
dist_{\mathcal{H}_\ep}(\mathbb{U}_\ep,\mathbb{U}_0) =  \sup_{a\in\mathcal{U}_\ep}\inf_{b\in\mathcal{L}\mathcal{U}_0}\|a-b\|_{X_\ep}.
\]
Fix $t\in[0,T]$ and $\alpha\in W_\ep$ so that $a=S_\ep(t)\alpha\in\mathcal{U}_\ep$. Then 
\begin{align}
\inf_{b\in\mathcal{L}\mathcal{U}_0}\|a-b\|_{X_\ep} & = \inf_{\substack{\tau\in [0,T] \\ \theta\in W_0}}\|S_\ep(t)\alpha-\mathcal{L}S_0(\tau)\theta\|_{X_\ep}  \notag \\
& \leq \inf_{\theta\in W_0}\|S_\ep(t)\alpha-\mathcal{L}S_0(t)\theta\|_{X_\ep}.  \notag
\end{align}
Since $S_\ep(t)\alpha=a$, 
\begin{align}
\sup_{\alpha\in W_\ep}\inf_{b\in\mathcal{L}\mathcal{U}_0}\|S_\ep(t)\alpha-b\|_{X_\ep} & \leq \sup_{\alpha\in W_\ep}\inf_{\theta\in W_0}\|S_\ep(t)\alpha-\mathcal{L}S_0(t)\theta\|_{X_\ep}   \notag \\ 
& =dist_{X_\ep}(S_\ep(t)W_\ep,\mathcal{L}S_0(t)W_0)  \notag \\
& \leq \max_{t\in [0,T]}dist_{X_\ep}(S_\ep(t)W_\ep,\mathcal{L}S_0(t)W_0). \notag
\end{align}
Thus, 
\begin{align}
\sup_{t\in [0,T]}\sup_{\alpha\in W_\ep}\inf_{b\in\mathcal{L}\mathcal{U}_0} \|S_\ep(t)\alpha-b\|_{X_\ep} \leq \max_{t\in [0,T]}dist_{X_\ep}(S_\ep(t)W_\ep,\mathcal{L}S_0(t)W_0),  \notag
\end{align}
and 
\begin{align}
\sup_{a\in\mathcal{U}_\ep}\inf_{b\in\mathcal{L}\mathcal{U}_0}\|a-b\|_{X_\ep} & \leq \sup_{t\in [0,T]}\sup_{\alpha\in W_\ep}\inf_{b\in\mathcal{L}\mathcal{U}_0} \|S_\ep(t)\alpha-b\|_{X_\ep}  \notag \\
& \leq \max_{t\in [0,T]}dist_{X_\ep}(S_\ep(t) W_\ep,\mathcal{L}S_0(t)W_0)  \notag \\
& \leq \max_{t\in [0,T]} \sup_{\alpha\in W_\ep}\inf_{\theta\in W_0}\|S_\ep(t)\alpha-\mathcal{L}S_0(t)\theta\|_{X_\ep}.  \notag
\end{align}

The norm is then expanded
\begin{align}
\|S_\ep(t)\alpha-\mathcal{L}S_0(t)\theta\|_{X_\ep} \leq \|S_\ep(t)\alpha & - \mathcal{L}S_0(t)\Pi\alpha\|_{X_\ep}  \notag \\
& + \|\mathcal{L}S_0(t)\Pi\alpha-\mathcal{L}S_0(t)\theta\|_{X_\ep} \label{4-1}
\end{align}
so that by the assumption described in (\ref{robustness}), there is a constant $M>0$ such that for all $t\in[0,T]$ and for all $\alpha\in W_\ep$,
\[
\|S_\ep(t)\alpha-\mathcal{L}S_0(t)\Pi\alpha\|_{X_\ep} \leq M\ep^\rho.
\]
Expand the square of the norm on the right hand side of (\ref{4-1}) to obtain, for $\Pi\alpha=\Pi(\alpha_1,\alpha_2)=\alpha_1\in X_0$ and $\theta\in X_0$,
\begin{equation}\label{triangle-r}
\|\mathcal{L}S_0(t)\Pi\alpha-\mathcal{L}S_0(t)\theta\|^2_{X_\ep} = \|S_0(t)\Pi\alpha-S_0(t)\theta\|^2_{X_0} + \ep\|\mathcal{E}S_0(t)\Pi\alpha - \mathcal{E}S_0(t)\theta\|^2_Y.
\end{equation}
By the local Lipschitz continuity of $\mathcal{E}$ on $X_0$, and by the local Lipschitz continuity of $S_0$ on $X_0$, there is $L>0$, depending on $W_0$, but independent of $\ep$, such that (\ref{triangle-r}) can be estimated by 
\[
\|\mathcal{L}S_0(t)\Pi\alpha-\mathcal{L}S_0(t)\theta\|^2_{X_\ep} \leq L^2(1+\ep)\|\Pi\alpha-\theta\|^2_{X_0}.
\]
Hence, (\ref{4-1}) becomes
\[
\|S_\ep(t)\alpha-\mathcal{L}S_0(t)\theta\|_{X_\ep} \leq M\ep^\rho + L\sqrt{1+\ep}\|\Pi\alpha-\theta\|_{X_0}
\]
and 
\[
\inf_{\theta\in W_0}\|S_\ep(t)\alpha-\mathcal{L}S_0(t)\theta\|_{X_\ep} \le M\ep^\rho + L\sqrt{1+\ep}\inf_{\theta = \Pi\alpha}\|\Pi\alpha-\theta\|_{X_0}.
\]
Since $\Pi\alpha\in\Pi W_\ep=W_0$, then it is possible to choose $\theta\in W_0$ to be $\theta=\Pi\alpha$. Therefore, 
\[
dist_{X_\ep}(\mathbb{U}_\ep,\mathbb{U}_0) = \sup_{\alpha\in W_\ep}\inf_{\theta\in W_0}\|S_\ep(t)\alpha-\mathcal{L}S_0(t)\theta\|_{X_\ep} \leq M\ep^\rho.
\]
This establishes the upper-semicontinuity of the sets $\mathbb{U}_\ep$ in $X_\ep$. 
\end{proof}

\begin{remark}
The upper-semicontinuous result given in Theorem \ref{t:robustness} is reminiscent of robustness results (cf. \cite{GGMP05}) insofar as we obtain explicit control over the semidistance in terms of the perturbation parameter $\ep$. 
\end{remark}

\subsection{The upper-semicontinuity of the family of global attractors for the model problems}

The goal of this section is to show that the assumptions of Theorem \ref{t:robustness} are meet. 
The conclusion is that the family of global attractors for the model problem are upper-semicontinuous.

Concerning the notation of the previous section, here $X_0=\mathcal{H}_0$, $Y=L^2(\Gamma)\times L^2(\Gamma)$, and $X_\ep=X_0\times Y=\mathcal{H}_\ep$. 
Recall that by the continuous dependence estimate (\ref{cont-dep}), $S_0$ is locally Lipschitz continuous on $\mathcal{H}_0$. 
Define the projection $\Pi:\mathcal{H}_\ep\rightarrow\mathcal{H}_0$ by 
\[
\Pi(u,v,\delta,\gamma) = (u,v);
\]
thus, for every subset $E_\ep\subset\mathcal{H}_\ep$, $\Pi E_\ep=E_0\subset\mathcal{H}_0$. Define the canonical extension $\mathcal{E}:\mathcal{H}_0\rightarrow L^2(\Gamma)\times L^2(\Gamma)$ by, for all $(u,v)\in\mathcal{H}_0$ and for all $\ep\in(0,1]$,
\[
\mathcal{E}(u,v)=(\ep u,-v).
\]
Clearly, $\mathcal{E}$ is locally Lipschitz on $\mathcal{H}_0$. Then the lift map, $\mathcal{L}:\mathcal{H}_0\rightarrow\mathcal{H}_\ep$, is defined, for any bounded set $E_0$ in $\mathcal{H}_0$, by
\[
\mathcal{L}E_0:=\{(u,v,\delta,\gamma)\in\mathcal{H}_\ep:(u,v)\in E_0, \delta=\ep u, \gamma = -v \}.
\]
Recall that $v=u_t$ and $\gamma=\delta_t$ in distributions, so the transport-type boundary condition, $\partial_{\bf{n}}u + u_t = 0$, is obtained from the limit problem.

The main result in this section is 

\begin{theorem}  \label{upper}
Assume (\ref{reg-assf-2}), (\ref{assf-2}), and (\ref{reg-assf-3}) hold subject to (\ref{aic2}).
Let $\mathcal{A}_0$ denote the global attractor corresponding to Problem (T) and for each $\ep\in(0,1]$, let $\mathcal{A}_\ep$ denote the global attractor corresponding to Problem (A). 
The family of global attractors $(\mathbb{A}_\ep)_{\ep\in[0,1]}$ in $\mathcal{H}_\ep$ defined by
\[
\mathbb{A}_\ep=\left\{ \begin{array}{ll} \mathcal{A}_\ep & 0<\ep\leq 1 \\ \mathcal{L}\mathcal{A}_0 & \ep=0. \end{array} \right.
\]
is upper-semicontinuous in $\mathcal{H}_\ep$, with explicit control over semi-distances in terms of $\ep$. 
(Note: we are not claiming that $\mathcal{LA}_0$ is a global attractor for Problem (T) in $\mathcal{H}_\ep.$)
\end{theorem}

The proof of Theorem \ref{upper} will rely on a dissipation integral on the high-ordered boundary term $\delta_{tt}.$
This bound can be achieved with the following 

\begin{lemma}  \label{utt-bndry}
For all $\varphi_0=(u_0,u_1)\in\mathcal{A}_0$, the solution $\varphi(t)$ of Problem (T) satisfies, for all $t\ge0,$
\begin{equation*}  \label{u-reg-0}
\int_0^t \|u_{tt}(\tau)\|^2_{L^2(\Gamma)}d\tau \le Q(R_0),
\end{equation*}
where $R_0$ is the radius of the absorbing set $\mathcal{B}_0$ in $\mathcal{H}_0$.
\end{lemma}

\begin{proof}
The proof follows by repeating part of \cite[Proof of Lemma 3.16]{Gal&Shomberg15} with $u$ in place of $w$, $U$ in place of $h$, and $f$ in place of $\psi$.
Let $\varphi_0\in\mathcal{A}_0$ (recall, by Theorem \ref{t:exp-attr-d}, $\mathcal{A}_0\subset\mathcal{D}_0$).
Indeed, we begin by differentiating the equations (\ref{pde}) and (\ref{dybc}) with respect to $t$, and set $U=u_t.$ 
Then $U$ satisfies the equations
\begin{equation}  \label{pde-U}
\left\{ \begin{array}{ll}
U_{tt}+U_t-\Delta U+U+f'(u)U=0 & \text{in}\quad (0,\infty )\times \Omega, \\ 
\partial_{\bf{n}}U=-U_t & \text{on}\quad (0,\infty )\times \Gamma, \\ 
U(0)=u_1,\quad U_t(0)=-u_1+\Delta u_0-u_0-f(u_0) & \text{at}\quad \{0\}\times{\overline{\Omega}}.
\end{array}\right.
\end{equation}
Arguing exactly as in \cite[Proof of Lemma 3.16]{Gal&Shomberg15}, we arrive at the differential inequality (cf. \cite[Equation (3.66)]{Gal&Shomberg15}), which holds for almost all $t\ge0$,
\begin{equation}  \label{u-reg-1}
\frac{d}{dt}\Psi + C_1 \Psi + C_2\|U_t\|_{L^2(\Gamma)}^2 \le C_3(R_0)\left( \|U\| \Psi + 1 \right),
\end{equation}
for positive $C_1, C_2$ and where $C_3(R_0)>0$ depends on the bound for trajectories on $\mathcal{A}_0$ in $\mathcal{D}_0$; i.e., the radius of the absorbing set $\mathcal{B}_0$ given by $R_0$.
The functional given by
\begin{equation*}  \label{diff-w-2-1}
\Psi(t):=\|U_t(t)\|^2+\alpha\langle U_t(t),U(t)\rangle+\|U(t)\|_1^2+\langle f'(u(t))U(t),U(t)\rangle,
\end{equation*}
where $\alpha>0$ is some (small) constant, satisfies, for some constants $C_4,C_5>0$, 
\begin{equation}  \label{diff-equiv}
C_4\|(U(t),U_t(t))\|_{\mathcal{H}_0}^2 \le \Psi(t) \le C_5\|(U(t),U_t(t))\|_{\mathcal{H}_0}^2,
\end{equation}
for all $t\ge0.$
We know from \cite[Lemma 3.14]{Gal&Shomberg15}, that for all $\eta>0$, there exists $Q_\eta(\cdot)\sim\eta^{-1}$, such that, for all $t\ge s>0,$
\[
\int_s^t \|u_t(\tau)\|d\tau \le \eta(t-s)+Q_\eta(R).
\]
Hence, with the aid of the Gronwall type inequality given in Proposition \ref{GL}, we recover the exponential decay property (which in turn, is used to provide the existence of a compact absorbing set in $\mathcal{H}_0$); precisely, for all $t\ge0$ there holds
\begin{equation}  \label{u-reg-2}
\Psi(t)\le C_3(R)\left( \Psi(0)e^{C_1t/2}+1 \right),
\end{equation}
where, with (\ref{diff-equiv}) and (\ref{pde-U})$_3$, 
\begin{equation}  \label{u-reg-3}
\Psi(0) \le C(R).
\end{equation}
Returning to (\ref{u-reg-1}), this time not neglecting the term with $U_t$, we integrate (\ref{u-reg-1}) on $(0,t)$ and apply (\ref{diff-equiv}) and (\ref{u-reg-3}) to produce the desired bound (\ref{u-reg-0}). 
This completes the proof.
\end{proof}

\begin{remark}
The proof of Lemma \ref{utt-bndry} is where we absolutely need the regularity assumptions (\ref{reg-assf-2}) and (\ref{reg-assf-3}).
Otherwise, the existence of global attractors for Problem (T) and Problem (A) under less restrictive assumptions can be shown following the asymptotic compactness method by J. Ball \cite{Ball00,Ball04}.
\end{remark}

The remaining claim establishes the assumption made in equation (\ref{robustness}), but we restrict our attention to the acoustic boundary condition subject to the special initial conditions (\ref{aic2}).
The claim indicates that trajectories on $\mathcal{A}_0$ and $\mathcal{A}_\ep$, with the same initial data, may be estimated, on compact time intervals and in the topology of $\mathcal{H}_\ep$, by a constant depending on the radii of the absorbing sets $\mathcal{B}_\ep$ and by the perturbation parameter $\ep$.

\begin{lemma}  \label{compare}
Let $T>0$. 
There is a constant $\Lambda=\Lambda(R_1)>0$ (cf. (\ref{acoustic-radius})), independent of $\ep$, such that, for all $t\in[0,T]$ and for all $\zeta_0\in \mathcal{A}_\ep$,
\begin{equation}  \label{robust-7}
\|S_\ep(t)\zeta_0 - \mathcal{L}S_0(t)\Pi\zeta_0\|_{\mathcal{H}_\ep} \le \Lambda\ep^{1/2}.
\end{equation}
\end{lemma}

\begin{proof}
Let $u$ denote the weak solution of Problem (A) with the given data $\zeta_0=(u_0,u_1,\delta_0,\delta_1)\in \mathcal{A}_\ep$, and let $\bar u$ denote the weak solution of Problem (T) corresponding to the initial data $\Pi\zeta_0=(u_0,u_1)\in \mathcal{A}_0$. 
To compare Problem (T) with Problem (A) in $\mathcal{H}_\ep$, the boundary condition (\ref{dybc}) is rewritten as the system,
\begin{equation}  \label{system-d}
\left\{\begin{array}{l} \bar\delta_{t} = -\bar u_t \\ 
\bar\delta_t = \partial_{\bf{n}}{\bar{u}} \\
\bar\delta(0,\cdot) = u_0.
\end{array}\right.
\end{equation}
Observe that through the definition of the lift map, 
\begin{equation}  \label{extic}
(\bar\delta(0,\cdot),\bar\delta_t(0,\cdot))=\mathcal{E}(u(0,\cdot),u_t(0,\cdot))=(u_0,-u_1).
\end{equation}
 
Let $z=u-\bar u$ and $w=\delta-\bar\delta$; hence, $z$ and $w$ satisfy the system
\begin{equation}\label{z-difference}
\left\{\begin{array}{ll} z_{tt} + z_t - \Delta z + z + f(u) - f(\bar u) = 0 & \text{in} \quad (0,\infty)\times\Omega \\
z(0,\cdot)=0, \quad z_t(0,\cdot)=0 & \text{at} \quad \{0\}\times\Omega \\ 
\ep w_{tt} + w_t + w = -z_t - \bar\delta - \ep\bar\delta_{tt} & \text{on} \quad (0,\infty)\times\Gamma \\
w_t = \partial_{\bf{n}}z & \text{on} \quad (0,\infty)\times\Gamma \\
w(0,\cdot)=\ep\delta_0, \quad w_t(0,\cdot)=\delta_1+u_1 & \text{at} \quad \{0\}\times\Gamma. \end{array}\right.
\end{equation}
(The initial conditions above are obtained by taking the difference between (\ref{aic2}) and (\ref{ic}) adjoined with (\ref{extic}).)
Observe, the function $\bar\delta$ is determined by the solution of the transport equation (\ref{system-d}).
Multiply equation (\ref{z-difference})$_1$ by $2z_t$ in $L^2(\Omega)$ and multiply equation (\ref{z-difference})$_3$ by $2w_t$ in $L^2(\Gamma)$ whereby summing the results, to obtain, for almost all $t\ge0,$
\begin{align}
\frac{d}{dt} & \left\{ \|z\|^2_1 + \|z_t\|^2 + \|w\|^2_{L^2(\Gamma)} + \ep^2\|w_t\|^2_{L^2(\Gamma)} \right\} + 2\|z_t\|^2 + 2\|w_t\|^2_{L^2(\Gamma)}  \notag \\
&  = -2\langle f(u)-f(\bar u),z_t \rangle - 2\langle \bar\delta, w_t \rangle_{L^2(\Gamma)} - 2\ep\langle \bar\delta_{tt}, w_t \rangle_{L^2(\Gamma)}.  \label{robust-3}
\end{align}
The first product on the right-hand side is estimated using the local Lipschitz continuity of $f$, 
\begin{equation}  \label{robust-5}
2|\langle f(u)-f(\bar u),z_t \rangle| \leq C_\Omega\|z\|^2_1 + 2\|z_t\|^2,
\end{equation}
where $C_\Omega$ is due to the continuous embedding $H^1(\Omega)\hookrightarrow L^6(\Omega).$
Estimating the remaining two products on the right-hand side yields, with $0<\ep\le 1,$
\begin{align}
& 2\langle \bar\delta, w_t \rangle_{L^2(\Gamma)} + 2\ep\langle \bar\delta_{tt}, w_t \rangle_{L^2(\Gamma)}   \notag \\ 
& \le 2\|\bar\delta\|_{L^2(\Gamma)} \|w_t\|_{L^2(\Gamma)} + 2\ep\|\bar\delta_{tt}\|_{L^2(\Gamma)}\|w_t\|_{L^2(\Gamma)}  \notag \\
& \le \|\bar\delta\|^2_{L^2(\Gamma)} + (1+\ep)\|w_t\|^2_{L^2(\Gamma)} + \ep\|\bar\delta_{tt}\|^2_{L^2(\Gamma)}.  \notag
\end{align}
Because $\bar\delta$ is obtained from the solution of the transport equation (\ref{system-d}), $\|\bar\delta\|_{L^2(\Gamma)}\le \ep\cdot \|\bar\delta_0\|_{L^2(\Gamma)}$ (observe, the dependence on $\ep$ is earned through the initial condition (\ref{system-d})$_3$ and the very nature of the solution to the transport equation serving as the boundary condition to Problem (T)).
Since the global attractor $\mathcal{A}_0$ is bounded in $\mathcal{B}_0$, we of course know that $\|\bar\delta_0\|^2_{L^2(\Gamma)}$ is bounded, uniformly in $t$ and $\ep\in(0,1]$, by the radius of $\mathcal{B}_0$; i.e., $R_0$. 
For the term $\|\bar\delta_{tt}\|^2_{L^2(\Gamma)}$, differentiating equation (\ref{system-d})$_1$ with respect to $t$ yields,
\[
\bar\delta_{tt} = -\bar u_{tt} \quad\text{on}\quad\Gamma.
\]
A bound on the term $\|\bar u_{tt}\|^2_{L^2(\Gamma)}$ is given in Lemma \ref{utt-bndry} above. 
Thus, there is a constant $C=C(R_0)>0$, independent of $\ep$, such that, for all $t\ge 0$, 
\begin{align}
& 2\langle \bar\delta, w_t \rangle_{L^2(\Gamma)} + 2\ep\langle \bar\delta_{tt}, w_t \rangle_{L^2(\Gamma)}   \notag \\ 
& \le (1+\ep)\|w_t\|^2_{L^2(\Gamma)} + \ep\cdot C(R_0).  \label{robust-4}
\end{align}
Combining (\ref{robust-3})-(\ref{robust-4}), leads to the differential inequality, which holds for almost all $t\ge 0$,
\begin{align}
\frac{d}{dt} & \left\{ \|z\|^2_1 + \|z_t\|^2 + \|w\|^2_{L^2(\Gamma)} + \ep\|w_t\|^2_{L^2(\Gamma)} \right\}   \notag \\ 
& \le C_\Omega \|z\|^2_1 + \ep\|w_t\|^2_{L^2(\Gamma)} + \ep\cdot C(R_0).  \notag
\end{align}
Let $M_2=\max\{C_\Omega,1\}.$ 
Integrating with respect to $t$ in the compact interval $[0,T]$ yields,
\begin{align}
\|z(t)\|^2_1 & + \|z_t(t)\|^2 + \|w(t)\|^2_{L^2(\Gamma)} + \ep\|w_t(t)\|^2_{L^2(\Gamma)}   \notag \\
& \le e^{M_2 T} \left(\|z(0)\|^2_1 + \|z_t(0)\|^2 + \|w(0)\|^2_{L^2(\Gamma)} + \ep \|w_t(0)\|^2_{L^2(\Gamma)}\right) + \ep\cdot C(R_0,T).  \label{robust-6}
\end{align}
Because of the initial conditions given in (\ref{z-difference})$_2$ and (\ref{z-difference})$_5$, we have $z(0)=z_t(0)=0$,
\[
\|w(0)\|^2_{L^2(\Gamma)} = \ep^2\|\delta_0\|^2_{L^2(\Gamma)} \quad \text{and} \quad \ep\|w_t(0)\|_{L^2(\Gamma)} = \ep\|\delta_1+u_0+u_1\|_{L^2(\Gamma)}. 
\] 
Since the initial condition $\zeta_0=(u_0,u_1,\delta_0,\delta_1)$ belongs to the bounded attractor $\mathcal{A}_\ep$, both $\|w(0)\|_{L^2(\Gamma)}\le C(R_0)$ and $\|w_t(0)\|_{L^2(\Gamma)}\le C(R_0)$, for some constant $C=C(R_0)>0$, where $R_0$ is the radius of the bounded absorbing set $\mathcal{B}_\ep.$
Thus, inequality (\ref{robust-6}) can be written as
\begin{align}
\|z(t)\|^2_1 & + \|z_t(t)\|^2 + \|w(t)\|^2_{L^2(\Gamma)} + \ep^2\|w_t(t)\|^2_{L^2(\Gamma)}   \notag \\ 
& \le \ep \cdot C(R_0,T).  \label{final-1}
\end{align}
Therefore, we arrive at
\[
\|S_\ep(t)\zeta_0-\mathcal{L}S_0(t)\Pi\zeta_0\|^2_{\mathcal{H}_\ep} \le \ep \cdot C(R_0,T).
\]
This establishes equation (\ref{robust-7}) and finishes the proof.
\end{proof}

The final proof is a direct application of Theorem \ref{t:robustness} to the model problem; however, Theorem \ref{t:robustness} may actually be applied to any family of sets that are described by (\ref{set-family-0})-(\ref{set-family-2}). 

\begin{proof}[Proof of Theorem \ref{upper}]
Because of the invariance of the global attractors, setting $W_0=\mathcal{A}_0$ in equation (\ref{set-family-0}) and setting $W_\ep=\mathcal{A}_\ep$ in equation (\ref{set-family-1}) produces, respectively, $\mathcal{U}_0=\mathcal{A}_0$ and $\mathcal{U}_\ep=\mathcal{A}_\ep$.
\end{proof}

\begin{remark}
It may be interesting to note that Problem (T) and Problem (A) also admit a global attractor under weaker conditions.
Indeed, we could assume the nonlinear term $f\in C^1(\mathbb{R})$ satisfies the sign condition, 
\begin{equation}  \label{wkassf-2}
\liminf_{|s|\rightarrow\infty}\frac{f(s)}{s}>-1,
\end{equation}
as before, but now the weaker growth condition
\begin{equation}  \label{wkassf-1}
|f'(s)|\leq \ell(1+s^2)
\end{equation}
for some $\ell\geq 0$.
With only these assumptions, it is possible to show that the weak solutions of Problem (T) and Problem (A), viewed only under the assumptions (\ref{wkassf-2}) and (\ref{wkassf-1}), admit a global attractor bounded in $\mathcal{H}_0.$
Additionally, one enjoys the upper-semicontinuity result of the previous section. 
Since the semiflow admits a bounded absorbing set, the existence of a global attractor follows when we establish that the associated semiflows are weakly continuous and asymptotically compact.
Then by the theory of generalized semiflows by J. Ball (cf. \cite{Ball00,Ball04}), it follows that the semiflows $S_\ep$, $\ep\in[0,1]$, admit a global attractor $\mathcal{A}_\ep$ in the phase space $\mathcal{H}_\ep$.
(For more on this, see \cite[Section 4]{Frigeri10}.)
\end{remark}

\appendix
\section{}

In this section we include some useful results utilized by Problem (T) and Problem (A). 
The first result can be found in \cite[Lemma 2.7] {Belleri&Pata01}.

\begin{proposition}  \label{t:diff-ineq-1}
Let $X$ be an arbitrary Banach space, and $Z\subset C([0,\infty);X)$. 
Suppose that there is a functional $E:X\rightarrow\mathbb{R}$ such that, for every $z\in Z$,
\[
\sup_{t\geq 0} E(z(t))\geq -r \ \text{and} \ E(z(0))\leq R
\]
for some $r,R\geq 0$. 
In addition, assume that the map $t\mapsto E(z(t))$ is $C^1([0,\infty))$ for every $z\in Z$ and that for almost all $t\geq 0$, the differential inequality holds
\[
\frac{d}{dt} E(z(t)) + m\|z(t)\|^2_X \leq C,
\]
for some $m>0$, $C\geq 0$, both independent of $z\in Z$. Then, for every $\iota>0$, there exists $t_0\geq 0$, depending on $R$ and $\iota$, such that for every $z\in Z$ and for all $t\geq t_0$,
\[
E(z(t))\leq \sup_{\xi\in X}\{ E(\xi):m\|\xi\|^2_X\leq C+\iota \}.
\]
Furthermore, $t_0=(r+R)/\iota$.
\end{proposition} 

The following statement is a frequently used Gr\"{o}nwall-type inequality \cite[Lemma 5]{Pata&Zelik06} or \cite[Lemma 2.2]{Grasselli&Pata02}.

\begin{proposition}  \label{GL}
Let $\Lambda :\mathbb{R}_{+}\rightarrow \mathbb{R}_{+}$ be an absolutely continuous function satisfying 
\begin{equation*}
\frac{d}{dt}\Lambda (t)+2\eta \Lambda (t)\leq h(t)\Lambda(t)+k,
\end{equation*}
where $\eta >0$, $k\geq 0$ and $\int_{s}^{t}h(\tau )d\tau \leq \eta(t-s)+m$, for all $t\geq s\geq 0$ and some $m\geq 0$. Then, for all $t\geq 0$, 
\begin{equation*}
\Lambda (t)\leq \Lambda (0)e^{m}e^{-\eta t}+\frac{ke^{m}}{\eta }.
\end{equation*}
\end{proposition}

The following result is the so-called transitivity property of exponential attraction from \cite[Theorem 5.1]{FGMZ04}.

\begin{proposition}  \label{t:exp-attr}
Let $(\mathcal{X},d)$ be a metric space and let $S_t$ be a semigroup acting on this space such that 
\[
d(S_t m_1,S_t m_2) \leq C e^{Kt} d(m_1,m_2),
\]
for appropriate constants $C$ and $K$. Assume that there exists three subsets $M_1$,$M_2$,$M_3\subset\mathcal{X}$ such that 
\[
{\rm{dist}}_\mathcal{X}(S_t M_1,M_2) \leq C_1 e^{-\alpha_1 t} \quad\text{and}\quad{\rm{dist}}_\mathcal{X}(S_t M_2,M_3) \leq C_2 e^{-\alpha_2 t}.
\]
Then 
\[
{\rm{dist}}_\mathcal{X}(S_t M_1,M_3) \leq C' e^{-\alpha' t},
\]
where $C'=CC_1+C_2$ and $\alpha'=\frac{\alpha_1\alpha_1}{K+\alpha_1+\alpha_2}$.
\end{proposition}

\section*{Acknowledgements}

The author would like to thank the anonymous referees for their careful reading of the manuscript and for their very helpful comments and suggestions.

\bigskip


\providecommand{\bysame}{\leavevmode\hbox to3em{\hrulefill}\thinspace}
\providecommand{\MR}{\relax\ifhmode\unskip\space\fi MR }
\providecommand{\MRhref}[2]{%
  \href{http://www.ams.org/mathscinet-getitem?mr=#1}{#2}
}
\providecommand{\href}[2]{#2}

\end{document}